\documentclass[12pt]{article}

\usepackage{epsfig}
\usepackage[utf8]{inputenc}
\usepackage{fullpage}
\usepackage{color}
\usepackage{amsmath}
\usepackage{setspace}
\usepackage{amsfonts}
\usepackage{amssymb}
\usepackage{graphicx}
\usepackage{hyperref}

\newtheorem{theorem}{Theorem}[section]
\newtheorem{lemma}[theorem]{Lemma}
\newtheorem{corollary}[theorem]{Corollary}
\newtheorem{definition}[theorem]{Definition}

\newenvironment{proof}{
    \noindent {\it Proof.}}{\hfill$\Box$
}

\DeclareMathSymbol{\N}{\mathbin}{AMSb}{"4E}
\DeclareMathSymbol{\Z}{\mathbin}{AMSb}{"5A}
\DeclareMathSymbol{\R}{\mathbin}{AMSb}{"52}

\begin{document}

\title{Well-posedness of Wasserstein Gradient Flow Solutions of Higher Order Evolution Equations}
\author{ Ehsan Kamalinejad \\
\small Department of Mathematics,
\small University of Toronto, \\
\small {\tt ehsan@math.toronto.edu}\\
}
\maketitle

\begin{abstract}

A relaxed notion of displacement convexity is defined and used to establish short time existence and uniqueness of Wasserstein gradient flows for higher order energy functionals. As an application, local and global well-posedness of different higher order degenerate non-linear evolution equations are derived. Examples include the thin-film equation and the quantum drift diffusion equation in one spatial variable.\\

\medskip\noindent{\bf Keywords:} optimal transport; Wasserstein gradient flows; displacement convexity; minimizing movement; well-posedness;  thin-film equation; higher order non-linear degenerate equations.

\end{abstract}

AMS subject classification: 35A15, 35K30, 76A20, 35K65.


\section{Introduction}

In the last decade, the theory of gradient flows in the Wasserstein space has been a rapidly expanding area of research. With a wide range of applications to evolution equations and functional inequalities, this theory has received an extensive amount of interest. In this section we start by recalling some historical backgrounds of the theory and then we state the summary of our result. For a comprehensive discussion of all aspects of the theory we refer the reader to monographs \cite{AmbrosioBook} and \cite{Villani}.\\


\subsection{Historical background}

The Wasserstein space $\mathcal{P}_2(\R^m)$ consists of the Borel probability measures on $\R^m$ with finite second moment. The quadratic optimal transport distance, also known as the {\bf Wasserstein distance} $W_2$, defines a distance function between any pair of measures $\mu,\nu \in \mathcal{P}_2(\R^m)$ given by
\begin{equation}\label{def:Wasserstein metric}
W_2(\mu,\nu):=\inf_{\gamma}\left\lbrace \int_{\R^m \times \R^m} |x-y|^2 \, d\gamma~:~ \gamma \in \Gamma(\mu,\nu) \right\rbrace^{\frac{1}{2}}
\end{equation}
where $\Gamma(\mu,\nu)\subset \mathcal{P}_2(\R^m \times \R^m)$ is the space of probability measures with marginals $\mu$ and $\nu$. We will refer to such measures $\gamma$ as {\bf transport plans}.\\

It turns out that $\mathcal{P}_2(\R^m)$ has a rich geometric structure and a formal Riemannian calculus can be performed on this space. The first appearance of the Riemannian calculus on $\mathcal{P}_2(\R^m)$ is due to Otto et al in \cite{JKO} and \cite{OttoPorous}. It was shown in \cite{OttoPorous} that the solution of the porous medium equation $\partial_t u =\Delta u^m$ can be reformulated as the gradient flow of the energy $E(u)=\int\frac{u^{m+1}}{m+1}$ on the Wasserstein space. Since then, the interaction between the Riemannian space $\mathcal{P}_2(\R^m)$ as a geometric object and evolution equations as analytic objects have attracted a lot of attention. This point of view is commonly called \emph{"Otto calculus"}.\\

A notion which has been very important in the developement of this theory is the notion of \emph{displacement convexity}. McCann in his thesis \cite{McCann-Thesis} introduced the notion of displacement convexity of an energy functional on the Wasserstein space. Under the displacement convexity assumption, he proved existence and uniqueness of minimizers of wide classes of energies, commonly referred to as potential, internal, and interactive energies. Displacement convexity had been defined before the development of the Wasserstein gradient flows, but after establishment of the Riemannian structure of the Wasserstein space, it turned out that displacement convexity can be interpreted as the standard convexity along the geodesics of the Wasserstein space. The displacement convexity condition, with its generalization to $\lambda$-displacement convexity, has a central role in existence, uniqueness, and long-time behaviour of the gradient flow of an energy functional.\\

Another important notion in the theory of Wasserstein gradient flows is the notion of \emph{minimizing movement}. Many of the rigorous proofs of the Wasserstein gradient flows are based on the method of minimizing movement. The minimizing movement scheme was suggested by De Georgi as a variational approximation of gradient flows in general metric spaces \cite{DeGiorgiMM}. It was later used by Jordan, Kinderlehrer, Otto \cite{JKO} and by Ambrosio, Savare, Gigli \cite{AmbrosioBook} to construct a systematic rigorous theory of Wasserstein gradient flows. This theory was soon used by many researchers to develop existence, uniqueness, stability, long time behaviour, and numerical approximation of evolution PDEs such as in \cite{AmbrosioZamboti}, \cite{WassersteinDiffusion}, \cite{Carrillo}, \cite{GangboCarlen}, \cite{Mccann02}, \cite{Numerics02}, \cite{Numerical}, and \cite{StickyParticles}.\\


\subsection{Summary of the results and outline of the paper}

In recent years, it has become apparent that Otto calculus also applies to higher-order evolution equations, at least on a formal level. The best-studied example is the thin-film equation $\partial_t u =- \nabla \cdot (u \nabla \Delta u)$, which corresponds to the gradient flow of the Dirichlet energy $E(u)=\frac{1}{2}\int |\nabla u|^2\, dx$. The hope is that gradient flow methods might help to resolve long-standing problems concerning well-posedness and long-time behaviour of this PDE. However, taking advantage of the gradient flow method has proved difficult. The main obstruction has been the lack of displacement convexity of the Dirichlet energy. The same problem arises for studying other energy functionals containing derivatives of the density. In \cite[open~problem~5.17]{VillaniTopics} Villani raised the question whether there is any useful example of a displacement convex functional that contains derivatives of the density. In \cite{Slepcev}, Carrillo and Slep{\v{c}}ev answered this question by providing a class of displacement convex functionals. Therefore it was proved that there is no fundamental obstruction for existence of such energies. However because of the lack of displacement convexity, the Wasserstein gradient flow method has not been very successful in studying gradient flows of the Dirichlet energy and other interesting energies of higher order.\\

Our result can be summarized as follows:
\begin{itemize}
\item We introduce a relaxed notion of $\lambda$-displacement convexity of an energy functional and in Theorem \ref{theorem1} we prove that, under this relaxed assumption, the general theory of well-posedness of Wasserstein gradient flows holds at least locally.
\item In Theorem \ref{the:convexity}, we prove that the Dirichlet energy, which is not $\lambda$-displacement convex in the standard sense, satisfies the relaxed version of $\lambda$-displacement convexity on positive measures. Hence the gradient flow of the Dirichlet energy is locally well-posed and the solution of the thin-film equation with positive initial data exists and is unique as long as positivity is preserved.
\item We show that the method developed to study thin-film equation applies to a range of PDEs of higher order and different forms.
\end{itemize}

The paper is organized as follows. After recalling the backgrounds of the theory, in Section 2.2 we define the new version of $\lambda$-displacement convexity which we call \emph{restricted $\lambda$-convexity}. Setting minor technicalities aside, the idea of restricted $\lambda$-convexity can be summarized in two simple principles: Firstly, the modulus of convexity, $\lambda$, can vary along the flow. Secondly, one can study $\lambda$-convexity \emph{locally on sub-level sets of the energy}. Note that the local analysis of gradient flows most likely fails without the help of energy dissipation. For example, the Dirichlet energy is not even locally $\lambda$-convex, because an arbitrarily small neighbourhood of a smooth positive measure contains measures with infinite energy where $\lambda$-convexity fails altogether. Instead, by taking advantage of the defining properties of the gradient flow, we study the flow on energy sub-level sets. The key observation is that typically finiteness of the energy implies some regularity on the measure which helps to elevate the formal calculations to rigorous proofs. For example, in 1-D, densities of finite Dirichlet energy lie in $H^1$.\\

After defining restricted $\lambda$-convexity, we state our first result, Theorem \ref{theorem1}. In this theorem we prove that if an energy functional is restricted $\lambda$-convex at a point $\mu$, then the corresponding gradient flow trajectory starting from $\mu$ exists and is unique at least for a short time. The proof is based on convergence of the minimizing movement scheme and the subdiffrential property that is carried over to the limiting curve. It is interesting that both of the constraints ''locality'' and ''energy boundedness'' are already encoded in the definition of the minimizing movement scheme (\ref{GMM}).\\

In Section 3 we apply the theory developed in the previous section to the Dirichlet energy. We prove that the Dirichlet energy on $\R/\Z$, is restricted $\lambda$-convex on the measures with positive density. This theorem re-derives the existing theory \cite{Bertozzi} of well-posedness of positive solutions of the thin-film equation by a direct geometric proof. To the best of our knowledge, this is the first well-posedness result for the thin-film equation based on Wasserstein gradient flows. Two key ideas are very useful in the proofs of this section: Firstly, the Wasserstein convergence and the uniform convergence are equivalent on energy sub-level sets. Secondly, finiteness of the energy can be used directly in the calculations of the second derivative of the energy along geodesics.\\

In the final section, we show that the method developed in Sections 2 and 3 can be applied to a wide class of energies of different forms and of higher orders. Some important examples have been studied using this method such as equations of higher order of the form $\partial_t u = (-1)^{k}\partial_x(u \partial_x^{2k+1}u)$, and equations of different forms, for instance the quantum drift diffusion equation $\partial_t u =-\partial_x (u \partial_x \frac{\partial_x^2 \sqrt{u}}{\sqrt{u}})$.\\

The Wasserstein gradient flow approach to PDEs has some interesting features. For example, it has a unified notion of solution which allows for very weak solutions and it is applicable to equations of higher order even with the lack of maximum principle. Also the minimizing movement scheme is a constructive method. Hence the proofs are constructive and one can derive numerical approximations based on the Wasserstein gradient flows similar to what has been done in \cite{Numerics02} and \cite{Numerical}.\\


\section{Well-posedness of the gradient flow}

In this section, we study the well-posedness problem of gradient flows on the Wasserstein space. Informally stated, a gradient flow evolves  by the steepest descent of an energy functional. This idea can be formalized in several different ways, some of which carry over to general metric spaces. Here, we consider the Fr\'echet subdifferential formulation of gradient flows. We will identify conditions on the energy functional that guarantee short-time existence and uniqueness.  The proof is based on a careful analysis of the minimizing movement scheme.\\

Let us recall the notion of a gradient flow on a finite dimensional Riemannian manifold. The ingredients of a gradient flow consist of three parts: a smooth manifold $M$, a metric $g$, and an energy $E$. Then the gradient flow of the energy $E$ can be formulated as
\begin{equation*}
\begin{cases}
\partial_t x_t = V_t &\text{(velocity vector)} \\
V_t = -\nabla E(x_t)  &\text{(steepest descent)}.
\end{cases}
\end{equation*} 
Note that the role of the metric is to convert the co-vector $dE=\frac{dE}{dt} dx$ into the corresponding vector $\nabla E$ on the tangent space.\\
  
In the case of Wasserstein gradient flows, the ingredients are given by: $\mathcal{P}_2(\R^m)$ as the manifold, the Wasserstein distance (and its infinitesimal version) as the metric, and an energy functional as the energy. The formulation of a Wasserstein gradient flow is given in (\ref{equ:flow}) and as it can be seen, it is similar to its finite dimensional counterpart.\\


\subsection{Geometry of the Wasserstein space}

In this part we gather some basic elements of the Riemannian structure of the Wasserstein space $\mathcal{P}_2(\R^m)$. Here, we work at a formal level  and we refer the reader to \cite{AmbrosioBook} or \cite{VillaniTopics} for rigorous proofs. It is worth mentioning that we consider the Euclidean space $\R^m$ as the underlying space of probability measures but one can replace $\R^m$ with any Hilbert space by slight modifications to the definitions as it is done in \cite{AmbrosioBook}.\\

Consider the Wasserstein distance (\ref{def:Wasserstein metric}). The Brenier-McCann theorem \cite{BrenierPolar} asserts that the minimum is always assumed and the minimal transport plan is concentrated on a graph of a map $T_{\mu}^{\nu}:\mathbb{R}^m \longrightarrow \mathbb{R}^m$, provided that $\mu \in \mathcal{P}_2^a(\R^m)$ where $\mathcal{P}_2^a(\R^m)$ is the set of absolutely continuous probability measures with respect to the Lebesgue measure. In this case, we have $\nu=(T_{\mu}^{\nu})_{\#}\mu$, and one can rewrite the Wasserstein distance as
\begin{equation}\label{W2distance}
W_2(\mu,\nu)=\left( \int_{\R^m} \left| T_{\mu}^{\nu}-Id \right|^2  d\mu \right)^{\frac{1}{2}}.
\end{equation}
Assuming $\mu=udx$ and $\nu=vdx$ are absolutely continuous measures, one can use the change of measure formula, given by the \textbf{Monge–Ampère} equation, to write an explicit relation between the densities $u$ and $v$ in terms of the optimal map:
\begin{equation}\label{Monge}
v(T_{\mu}^{\nu}(x))=\frac{u(x)}{det(DT_{\mu}^{\nu})(x)}.
\end{equation} 
The optimal transport map also defines the \textbf{geodesic $\mu_s$} between two measures $\mu_0$ and $\mu_1$ given by the pushforward of the linear interpolation between the optimal map $T_{\mu_0}^{\mu_1}$ and the identity map: 
\begin{equation}\label{eq:geodesic}
\mu_s=\left( (1-s)Id + sT_{\mu_0}^{\mu_1} \right)_{\#}\mu_0.
\end{equation}

The appropriate class of curves inside $\mathcal{P}_2(\R^m)$ which have a natural notion of tangent to them turns out to be the class of absolutely continuous curves $AC_{loc}^2([0,\infty); \mathcal{P}_2(\R^m))$.\\
A curve $\mu_t$ belongs to $AC_{loc}^2([0,\infty); \mathcal{P}_2(\R^m))$ if there exist a locally $L^2(dt)$ integrable function $g$ such that
\[ W_2(\mu_a,\mu_b) \leqslant \int_a^b g(t) dt ~~~\forall a,b \in [0,\infty). \]
The absolutely continuous curves are given by mass conservative deformations of the measures i.e. they satisfy the \textbf{continuity equation}:
\begin{equation*}
\partial_t\mu_t+\nabla.(\mu_t V_t)=0
\end{equation*}
for a \textbf{velocity vector field} $V_t$ of deformations of $\mu_t$. This equation is assumed to hold in the distributional sense. In \cite{Benamou} Brenier and Benamou showed that $\mathcal{P}_2(\R^m)$ is a length space in the sense that the distance of two measures is given by the length of the shortest path between them:
\begin{equation}\label{v01}
W_2(\mu,\nu)=\inf \left\lbrace \int_0^1 ( \int_{\R^m}|V_t|^2d\mu_t )^{\frac{1}{2}} dt ~~\text{s.t.}~~ \partial_t\mu_t + \nabla . (\mu_t V_t) =0; ~ \mu_0=\mu,~\mu_1=\nu \right\rbrace
\end{equation}
where the infimum is taken over all curves in $AC^2([0,1]; \mathcal{P}_2(\R^m))$. For a given curve $\mu_t \in AC^2([0,1]); \mathcal{P}_2^a(\R^m))$ there might be many velocity vectors that satisfy the same continuity equation $\partial_t \mu_t+\nabla .(\mu_t V_t)=0$. For instance vector fields of the form $F_t+V_t$ where $\nabla .(\mu_t F_t)=0$ all satisfy the same  continuity equation. However there is a unique vector field that minimizes (\ref{v01}), i.e. the one which defines the distance between $\mu$ and $\nu$. This optimal velocity vector field is defined to be the \textbf{tangent vector field} to the curve $\mu_t$. The tangent vector field of $\mu_t$ can also be expressed in term of the optimal maps. If $V_t$ is the tangent vector field of $\mu_t$ then 
\begin{equation*}
V_t=\lim_{\epsilon \to 0} \dfrac{T_{\mu_t}^{\mu_{t+\epsilon}}-Id}{\epsilon}.
\end{equation*}
The converse is also true, i.e. for a given optimal map $T_{\mu}^{\nu}$, the vector field $T_{\mu}^{\nu}-Id$ is a tangent vector at $\mu$ for some curve that passes $\mu$. The tangent vector fields are also useful in calculating the derivative of the Wasserstein metric along curves. Let $\mu_t \in AC_{loc}^2(\R^+;\mathcal{P}_2^a(\R^m))$. By \cite[Chapter 8]{AmbrosioBook} \textbf{the derivative of the Wasserstein metric} along the curve $\mu_t$ is given by
\begin{equation}\label{derivative of metric}
\dfrac{d}{dt}W_2(\mu_t,\nu)^2= 2 \int_{\R^m}\langle V_t, Id - T_{\mu_t}^{\nu} \rangle d\mu_t ~~~~ \forall \nu \in \mathcal{P}_2(\R^m)
\end{equation}
where $V_t$ is the tangent vector field to $\mu_t$ and $\langle.,.\rangle$ is the standard inner product on  $\R^m$.\\

The Wasserstein metric is closely related to a certain weak topology on ${\cal P}_2(\R^m)$, induced by \textbf{narrow convergence}:
\begin{equation}\label{narrow convergence}
\mu_n \xrightarrow{\text{narrow}} \mu ~~\Longleftrightarrow ~~ \int_{\R^m} fd\mu_n \rightarrow \int_{\R^m} fd\mu ~~~~\forall f\in C_b^0(\R^m)\,,
\end{equation} 
where $C_b^0(\R^m)$ is the set of of continuous bounded real functions on $\R^m$. The topologies induced by the narrow convergence and the Wasserstein distance are equivalent for sequences of measures with uniformly bounded second moments:
\begin{equation}\label{narrow}
\lim_{n \to \infty} W_2(\mu_n,\mu)=0 ~~\Longleftrightarrow ~~ \begin{cases} \mu_n \xrightarrow{\text{narrow}} \mu\\
\lbrace \mu_n \rbrace ~~\text{has uniformly bounded 2-moments}. \end{cases}\\
\end{equation}


\subsection{Wasserstein gradient flows}

Now we describe gradient flows on the Wasserstein space. Consider the energy functional $E:\mathcal{P}_2(\R^m) \rightarrow [0,\infty]$ and let its domain, $D(E)$, be the set where $E$ is finite. Let $\mu$ lie in $D(E) \cap \mathcal{P}_2^a(\R^m)$. A vector field $\xi \in L^2(d\mu)$ belongs to the {\bf subdifferential} of $E$ at $\mu$ if
\begin{equation}\label{def:subdiff}
\liminf_{\substack{
            \nu \to \mu \\
            \nu \in D(E)}} \dfrac{E(\nu)-E(\mu)- \int_X \left\langle \xi, T_{\mu}^{\nu}-Id \right\rangle \,\mathrm{d}\mu}{W_2\left( \mu,\nu \right)} \geqslant 0.
\end{equation}

We say that a curve $\mu_t \in AC^2_{loc}(\R^+,\mathcal{P}_2^a(\R^m)) $ is a trajectory of the {\bf gradient flow} for the energy $E$, if there exists a velocity field $V_t$ with $|V_t|_{L^2(d\mu_t)}\in L^1_{loc}(\R^+)$ such that 
\begin{equation}\label{equ:flow}
\begin{cases}
\partial_t\mu_t + \nabla \cdot\left( \mu_t V_t \right)=0 
&\text{(continuity equation)}, \\
V_t \in -\partial E(\mu_t)  &\text{(steepest descent)}
\end{cases}
\end{equation}
hold for almost every $t>0$. We will refer to (\ref{equ:flow}) as the {\bf gradient flow equation}. The continuity equation, which is assume to hold in the distributional sense, links the curve with its velocity vector field and ensures that the mass is conserved. The steepest descent equation expresses that the gradient flow evolves in the direction of maximal energy dissipation.\\

Next we describe the link between Wasserstein gradient flows and evolution PDEs. Let $\mu=udx$ be in $D(E)$ and let $V\in \partial E(\mu)$ be a tangent vector field at $\mu$. Consider a linear perturbation of $\mu$ given by the curve $\mu_{\epsilon}:=(Id+\epsilon W)_{\#}\mu$ for small values of $\epsilon>0$ where $W\in C^{\infty}_c(\R^m;\R^m)$. By the subdifferential inequality (\ref{def:subdiff}) we have
\begin{equation*}
\limsup_{\epsilon \uparrow 0} \dfrac{E(u_{\epsilon})-E(u)}{\epsilon}\leqslant \int_{\R^m} \langle V,W \rangle udx
\leqslant \liminf_{\epsilon \downarrow 0} \dfrac{E(u_{\epsilon})-E(u)}{\epsilon}.
\end{equation*}
On the other hand, assuming $C^2$ regularity on $u_{\epsilon}$ and using standard first order variations we have
\begin{equation*}
\lim_{\epsilon \to 0} \dfrac{E(u_{\epsilon})-E(u)}{\epsilon}=\int_{\R^m} (\dfrac{\delta E(u)}{\delta u}) (\frac{\partial u_{\epsilon}}{\partial \epsilon})udx
\end{equation*}  
where $\dfrac{\delta E(u)}{\delta u}$ stands for standard first variations of $E$. Therefore 
\begin{equation*}
\int_{\R^m} \langle V,W \rangle udx = \int_{\R^m} \dfrac{\delta E(u)}{\delta u} \partial_{\epsilon} u_{\epsilon} dx.
\end{equation*}
The continuity equation for the curve $u_{\epsilon}$ implies that $\partial_{\epsilon} u_{\epsilon} =  - \nabla.(uW)$. Hence
\begin{equation*}
\int_{\R^m} \langle V,W \rangle udx = -\int_{\R^m} \lbrace \dfrac{\delta E(u)}{\delta u} \nabla .(uW) \rbrace dx.
\end{equation*}
Integrating by parts, we have
\begin{equation*}
\int_{\R^m} \langle V,W \rangle udx = \int_{\R^m} \langle \nabla (\dfrac{\delta E(u)}{\delta u}), W \rangle udx.
\end{equation*}
Since $W$ is arbitrary we have
\begin{equation}\label{v04}
V(x) =\nabla (\dfrac{\delta E(u)}{\delta u})(x)~~~\text{for $\mu$-a.e. $x$}.
\end{equation}
Now assume that a curve $\mu_t=u_tdx$ satisfies the gradient flow equation (\ref{equ:flow}). Steepest descent equation and (\ref{v04}) imply
\begin{equation*}
V_t(x) =- \nabla (\dfrac{\delta E(u)}{\delta u})(x).
\end{equation*}
By plugging $V_t$ into continuity equation, we have
\begin{equation} \label{v05}
\partial_t u = \nabla . \left(u \nabla (\dfrac{\delta E(u)}{\delta u}) \right).
\end{equation}
This is the corresponding PDE for the gradient flow of the energy $E$. For example in the case of the Dirichlet energy $E(u)=\int_{\R^m}|\nabla u|^2 dx$, the first variation is given by $\frac{\delta E(u)}{\delta u}=-\Delta u$. Therefore the corresponding PDE is the \textbf{thin-film equation}:
\begin{equation*}
\partial_t u = - \nabla .(u \nabla \Delta u).
\end{equation*}

Our proofs are based on the \textbf{minimizing movement scheme} as a discrete-time approximation of a gradient flow which is described here. Let $\mu_0 \in D(E)$, and fix the step size $\tau>0$. Recursively define a sequence $ \left\lbrace M_{\tau}^n \right\rbrace_{n=1}^{+\infty} $ by setting $M_0^\tau=\mu_0$, and for $n\geqslant 1$,
\begin{equation}\label{GMM}
M_{n}^{\tau} = \underset{\mu \in D(E)}{\text{\rm argmin}} 
\left\lbrace E (\mu) + \dfrac{1}{2\tau} W_2^2 \left( M_{n-1}^{\tau},\mu \right) \right\rbrace.
\end{equation}
The formal Euler-Lagrange equation for this minimization problem is given by 
\begin{equation*}
\ U_n^{\tau}\in -\partial E \left( M_n^{\tau} \right)
\end{equation*}
where $U_n^{\tau}=-\frac{T_{M^\tau_{n}}^{M^\tau_{n-1}}-Id}{\tau}$. Next we define a piecewise constant curve and a corresponding velocity field by
$$\mu^\tau_t:= M_n^{\tau}\,,\qquad 
V^{\tau}_t:=U_n^{\tau}\,,\qquad \mbox{for}\ (n-1)\tau <t\le n\tau \,.
$$
We have
\begin{equation}\label{sub of piecwise}
V^\tau_t \in -\partial E(\mu^\tau_t ) ~~~~ \forall t>0.
\end{equation}
This equation suggests that $\mu_t^{\tau}$ is an approximation of the gradient flow trajectory of $E$ starting from $\mu_0$.\\

There is a standard set of hypothesises that we assume throughout this section. We gather the hypothesises here:
\begin{itemize}
\item $E$ is nonnegative, and its sub-level sets are locally compact in the Wasserstein space.
\item $E$ is lower semicontinuous under narrow convergence.
\item $D(E) \subseteq \mathcal{P}_2^a(\R^m)$. 
\end{itemize}

The first two conditions guarantee existence and convergence of minimizing movement scheme (\ref{GMM}). The third condition ensures that measures of finite energy are absolutely continuous, allowing us to use transport maps for studying the Wasserstein distance which simplifies the calculations, and allow us to view the subdifferential as a tangent vector. Note that these conditions can be sharpened as in \cite{AmbrosioBook}, but to make the presentation more apparent, we prefer to work in this more concrete setting.\\

The main condition that guarantees existence and uniqueness of a Wasserstein gradient flow is given by the displacement convexity condition which asks for convexity of the energy along geodesics of the Wasserstein space. Let $\mu_s:[0,1]\longrightarrow\mathcal{P}_2(\R^m)$ be the geodesic between $\mu_0,\mu_1 \in \mathcal{P}_2(X)$. An energy functional $E$ is called \textbf{displacement convex} along $\mu_s$ if
\begin{equation*}
E(\mu_s)\leqslant (1-s)E(\mu_0)+sE(\mu_1)~~~~s \in [0,1].
\end{equation*}
More generally the energy is called \textbf{$\lambda$-displacement convex} or in short $\lambda$-convex if the convexity is bounded from below by the constant $\lambda$, i.e.
\begin{equation}\label{def:convx along curve}
E(\mu_s)\leqslant (1-s)E(\mu_0)+sE(\mu_1)-\dfrac{\lambda}{2}s(1-s)W_2^2(\mu_0,\mu_1)~~~~s \in [0,1].
\end{equation}
Furthermore, assuming that $E(\mu_s)$ is smooth as a function of $s$, one can write a derivative version of $\lambda$-convexity. In this case, $E$ is $\lambda$-convex along $\mu_s$ if
\begin{equation}\label{def:smooth conv along curve}
\dfrac{d^2}{ds^2}E(\mu_s)\geqslant \lambda W_2^2(\mu_0,\mu_1).
\end{equation}
An energy is called $\lambda$-convex if it satisfies (\ref{def:convx along curve}) along all geodesics of $\mathcal{P}_2(\R^m)$.\\


\subsection{Restricted $\lambda$-convexity and local well-posedness}

\begin{definition}[Restricted $\lambda$-convexity.]\label{def:restricted}
We say that an energy $E$ is restricted $\lambda$-convex at $\mu \in D(E)$ with $E(\mu)<c<+\infty$, if $\exists \delta>0$, such that $E$ is $\lambda$-convex along the geodesics connecting any pair of measures $\nu_1,\nu_2 \in B_{\delta}(\mu)\cap E_c$, where $B_{\delta}(\mu)=\left\lbrace \nu ~~s.t.~~ W_2(\mu,\nu)<\delta \right\rbrace$ and $E_c= \left\lbrace \nu ~~s.t.~~ E(\nu)<c \right\rbrace $.
\end{definition}

\begin{figure}[htb]
\centering
\includegraphics[width=0.4 \textwidth]{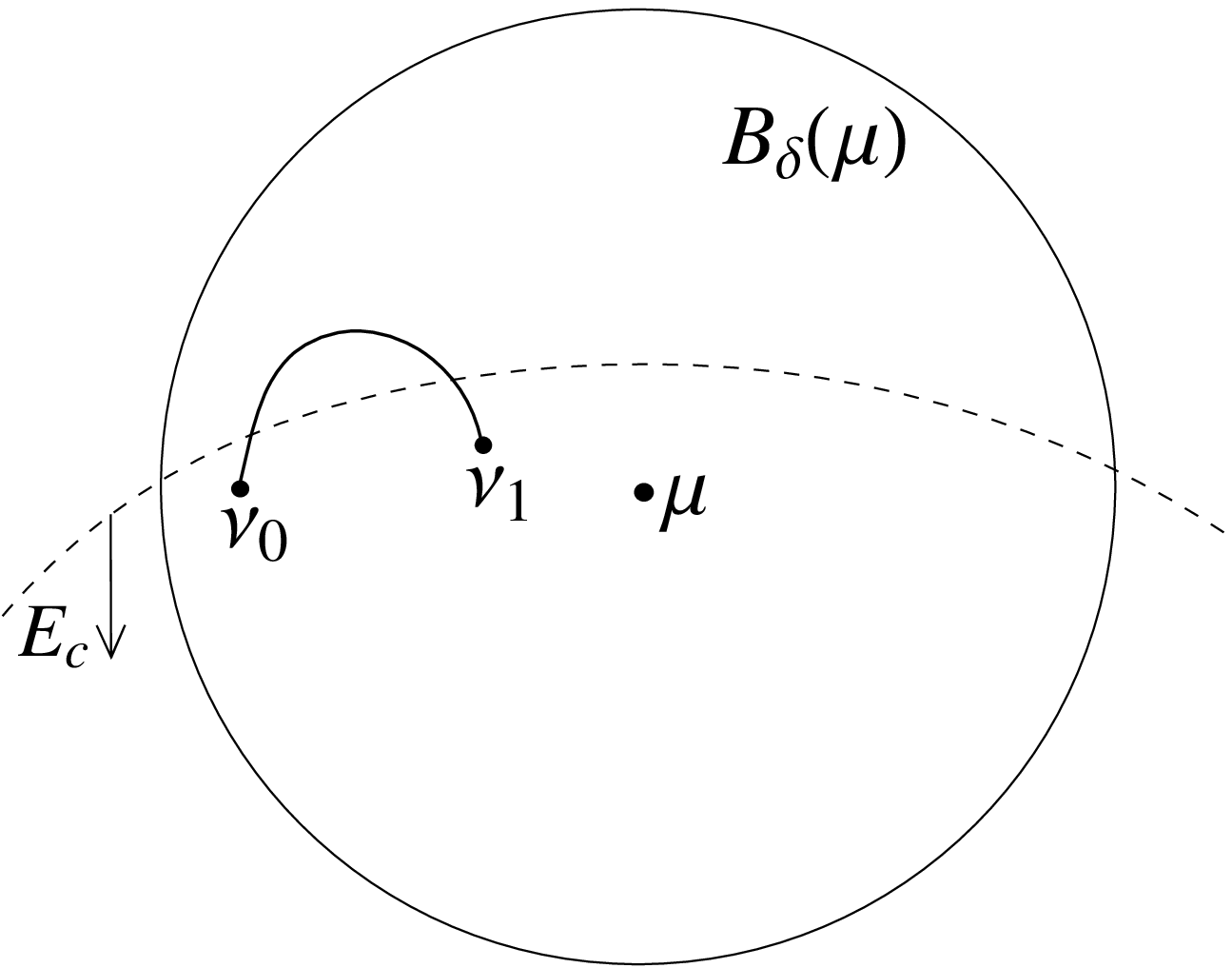}
\caption{Restricted $\lambda$-convexity}
\label{fig:Restricted}
\end{figure}


The following lemma has a key role in the arguments of Theorem \ref{theorem1}. It shows how one can use restricted $\lambda$-convexity assumption to study the subdifferential of an energy.

\begin{lemma}[Subdifferential and restricted $\lambda$-convexity.]\label{lem:subdif}
Assume that $E$ is restricted $\lambda$-convex at $\mu$. Then a vector field $\xi \in L^2(d\mu)$ belongs to the subdifferential of $E$ at $\mu$ if and only if 
\begin{equation}\label{equ:lambdasubdif}
E(\nu)-E(\mu)\geqslant \int_X \langle \xi, T_{\mu}^{\nu} -Id \rangle \,\mathrm{d}\mu + \dfrac{\lambda}{2}W_2^2(\mu,\nu) ~~~~\forall \nu \in B_{\delta}(\mu) \cap E_{c}
\end{equation}
where $B_{\delta}(\mu) \cap E_c$ is the corresponding restricted $\lambda$-convexity domain at $\mu$. 
\end{lemma}

\begin{proof}
First we claim that for studying the subdifferential of the functional, it is enough to consider the restricted domain $B_{\delta}(\mu) \cap E_{c}$. Let $\xi \in L^2(d\mu)$, we have to show that
\begin{equation}\label{equ:subdif}
\liminf_{\substack{
            \nu \to \mu \\
            \nu \in D(E)}} \dfrac{E(\nu)-E(\mu)-\int_X \langle \xi, T_{\mu}^{\nu} -Id  \rangle \,\mathrm{d}\mu}{W_2(\mu,\nu)} \geqslant 0.
\end{equation}
if and only if
\begin{equation}\label{equ:ressubdif}
\liminf_{\substack{
            \nu \to \mu \\
            \nu \in B_{\delta}(\mu) \cap E_{c}}}\dfrac{E(\nu)-E(\mu)-\int_X \langle \xi, T_{\mu}^{\nu} -Id  \rangle \,\mathrm{d}\mu}{W_2(\mu,\nu)} \geqslant 0.
\end{equation}
$\Downarrow$ is trivial.\\
For $\Uparrow$ assume that $\lbrace \nu_n \rbrace_1^{\infty}$ is a minimizing sequence for (\ref{equ:subdif}).\\
In the case that $\liminf_{n \to \infty} E(\nu_n)-E(\mu) > 0$ we have
\small
\begin{equation*}
\begin{aligned}
\dfrac{E(\nu_n)-E(\mu)-\int_X \langle \xi, T_{\mu}^{\nu_n} -Id  \rangle \,\mathrm{d}\mu}{W_2(\mu,\nu_n)} &\geqslant  \dfrac{E(\nu_n)-E(\mu)-\left( \int_X \left| \xi \right|^2 \,\mathrm{d}\mu \right)^{1/2} \left( \int_X \right| T_{\mu}^{\nu_n} -Id \left|^2 \,\mathrm{d}\mu \right)^{1/2}}{W_2(\mu,\nu_n)}\\
&=\dfrac{E(\nu)-E(\mu)}{W_2(\mu,\nu_n)} -\left( \int_X \left| \xi \right|^2 \,\mathrm{d}\mu \right)^{1/2}\dfrac{\left( \int_X \right| T_{\mu}^{\nu_n} -Id \left|^2 \,\mathrm{d}\mu \right)^{1/2}}{W_2(\mu,\nu_n)}\\
&=\left[ \dfrac{E(\nu)-E(\mu)}{W_2(\mu,\nu_n)} - \left| \xi \right|_{L_{\mu}^2}\right] \longrightarrow +\infty.
\end{aligned}
\end{equation*} 
\normalsize
Therefore inequality (\ref{equ:subdif}) is automatically true if $\liminf_{n \to \infty} E(\nu_n)-E(\mu) > 0$. Hence one only needs to consider sequences $\lbrace \nu_n \rbrace_1^{\infty}$ such that $\lim_{n \to \infty} E(\nu_n)-E(\mu) \leqslant 0$. Therefore, for large enough $n$ we have $E(\nu_n) < c$. On the other hand, $\nu_n \xrightarrow []{W_2} \mu$. Hence (\ref{equ:ressubdif}) and (\ref{equ:subdif}) are equivalent.\\

It is clear that (\ref{equ:lambdasubdif}) implies (\ref{equ:ressubdif}). Conversely, let $\xi \in L^2(d\mu)$ satisfy (\ref{equ:ressubdif}). Let $\nu \in B_{\delta}(\mu) \cap E_{c}$. Since $E$ is restricted $\lambda$-convex at $\mu$, we have $\lambda$-convexity of $E$ along the geodesic $\mu_s$ connecting $\mu$ to $\nu$. Therefore
\begin{equation*}
E(\mu_s) \leqslant (1-s) E(\mu) + s E(\nu) -\dfrac{\lambda}{2}s(1-s)W_2^2(\mu,\nu)  ~~~~\forall s\in [0,1] .
\end{equation*}
Dividing by $s$ and reordering, we have
\begin{equation}\label{equ:mus}
\dfrac{E(\mu_s) - E(\mu)}{s}\leqslant E(\nu)-E(\mu)-\dfrac{\lambda}{2}(1-s)W_2^2(\mu,\nu).
\end{equation}
$\xi$ is in the subdifferential of $E$ at $\mu$. Hence
\begin{equation}\label{equ:andq}
\begin{aligned}
\liminf_{s \to 0^+} \dfrac{E(\mu_s)-E(\mu)}{s} &\geqslant \lim_{s \to 0^+} \dfrac{1}{s} \int_{X} \langle \xi, T_{\mu}^{\mu_s}-Id \rangle \,\mathrm{d}\mu \\
&= \int_{X} \langle \xi, T_{\mu}^{\nu} -Id\rangle \,\mathrm{d}\mu
\end{aligned}
\end{equation}
where we used linearity of the interpolate map $T_{\mu}^{\mu_s}=Id+s(T_{\mu}^{\nu}-Id)$. Therefore (\ref{equ:mus}) and (\ref{equ:andq}) imply
\[
E(\nu)-E(\mu)\geqslant \int_X \langle \xi, T_{\mu}^{\nu} -Id \rangle \,\mathrm{d}\mu + \dfrac{\lambda}{2}W_2^2(\mu,\nu).
\]
\end{proof}\\


The following lemma is used in Theorem \ref{theorem1} when we study weak convergence of tangent vector fields. 

\begin{lemma}\label{lemma:convergence}
Let $\mu_t^k, \mu_t \in AC^2([0,\hat{t}];\mathcal{P}_2^a(\mathbb{R}^m))$ and let $V_t^k \in L^2(d\mu_t^k),~ V_t \in L^2(d\mu_t)$. Assume that
\begin{itemize}
\item $\mu_t^k \xrightarrow []{W_2} \mu_t$ uniformly on $[0,\hat{t}]$.
\item  $V_t^k$ \textbf{weakly converges} to  $V_t$ in the sense that $\forall U\in C_b^0([0,\hat{t}] \times \R^m)$, we have
\[ \lim_{k \to \infty} \int_0^{\hat{t}} \int_{\mathbb{R}^m} \left\langle V_t^k , U(t,x) \right\rangle d\mu_t^k dt =\int_0^{\hat{t}} \int_{\mathbb{R}^m} \left\langle V_t , U(t,x) \right\rangle d\mu_t dt. \]
\end{itemize}
Then $\forall \nu \in \mathcal{P}_2(\mathbb{R}^m)$
\begin{equation*}
\lim_{k \to \infty} \int_0^{\hat{t}} \int_{\mathbb{R}^m} \left\langle V_t^k , T_{\mu_t^k}^{\nu}-Id \right\rangle d\mu_t^k dt =\int_0^{\hat{t}} \int_{\mathbb{R}^m} \left\langle V_t , T_{\mu_t}^{\nu}-Id \right\rangle d\mu_t dt.
\end{equation*}
\end{lemma}

\begin{proof}
Since $V_t^k$ is weakly convergent by uniform boundedness principle we have\\
 $~\sup_{k}\int_0^{\hat{t}} \int_{\mathbb{R}^m} |V_t^k|^2 d\mu_t^k dt < +\infty$. Let 
\[M=\sup_k\int_0^{\hat{t}} (\int_{\mathbb{R}^m} \left| V_t^k \right|^2 d\mu_t^k + \int_{\mathbb{R}^m}\left| V_t \right|^2 d\mu_t )dt.\]
Choose $T_t\in C_c^0([0,\hat{t}] \times \R^m)$ such that 
\[\int_0^{\hat{t}} \int_{\mathbb{R}^m} |T_{\mu_t}^{\nu}-T_t|^2 d\mu_t dt <\epsilon^2.\]
We have
\begin{equation*}
\begin{aligned}
&~\left| \int_0^{\hat{t}} \int_{\mathbb{R}^m} \left\langle V_t^k , T_{\mu_t^k}^{\nu}-Id\right\rangle d\mu_t^k dt - \int_0^{\hat{t}} \int_{\mathbb{R}^m} \left\langle V_t , T_{\mu_t}^{\nu}-Id\right\rangle d\mu_t dt \right|\\
&\leqslant \underbrace{\left| \int_0^{\hat{t}} \int_{\mathbb{R}^m} \left\langle V_t^k,Id \right\rangle d\mu_t^k dt-\int_0^{\hat{t}} \int_{\mathbb{R}^m} \left\langle V_t,Id \right\rangle d\mu_t dt \right| }_A\\
&+ \underbrace{\left| \int_0^{\hat{t}} \int_{\mathbb{R}^m} \left\langle V_t^k,T_{\mu_t^k}^{\nu}-T_t \right\rangle d\mu_t^k dt \right|}_B\\
&+ \underbrace{\left| \int_0^{\hat{t}} \int_{\mathbb{R}^m} \left\langle V_t^k,T_t \right\rangle  d\mu_t^k dt - \int_0^{\hat{t}} \int_{\mathbb{R}^m} \left\langle V_t,T_{\mu_t}^{\nu} \right\rangle  d\mu_t dt \right|}_C
\end{aligned}
\end{equation*}

We study each of the items separately.\\

Since $\mu_t^k$ is uniformly converging to $\mu_t$, the second moment of $\mu_t^k$ is uniformly bounded. In particular there is a compact set $S \subset [0,\hat{t}] \times \R^m $ such that
\begin{equation}\label{h01}
(\int_{S^c}|x|^2 \mu_t dt + \sup_{k}\int_{S^c}|x|^2 \mu_t^k dt) < \epsilon^2.
\end{equation}
We have
\begin{equation*}
\begin{aligned}
\lim_{k \to \infty} A &= \lim_{k \to \infty} \left| \int_0^{\hat{t}} \int_{\mathbb{R}^m} \left\langle V_t^k,Id \right\rangle d\mu_t^k dt-\int_0^{\hat{t}} \int_{\mathbb{R}^m} \left\langle V_t,Id \right\rangle d\mu_t dt \right|\\
&\leqslant \lim_{k \to \infty}\left| \int_S  \left\langle V_t^k,Id \right\rangle  d\mu_t^k dt-\int_S \left\langle V_t,Id \right\rangle  d\mu_t dt \right|\\
&+\lim_{k \to \infty} \left| \int_{S^c}  \left\langle V_t^k,Id \right\rangle d\mu_t^k dt-\int_{S^c} \left\langle V_t,Id \right\rangle  d\mu_t dt \right|.
\end{aligned}
\end{equation*}
Because $S$ is compact one can use weak convergence of $V_t^k$ on $S$. Hence the limit of the first term vanishes and we have
\begin{equation*}
\begin{aligned} 
\lim_{k \to \infty} A & \leqslant \lim_{k \to \infty} \left| \int_{S^c}  \left\langle V_t^k,Id \right\rangle  d\mu_t^k dt-\int_{S^c} \left\langle V_t,Id \right\rangle d\mu_t dt \right|\\
&\leqslant \frac{\epsilon}{2} \lim_{k \to \infty} \left( \int_{S^c}  |V_t^k|^2 d\mu_t^k dt + \int_{S^c} |V_t|^2 d\mu_t dt \right)\\
&+\lim_{k \to \infty} \frac{1}{2\epsilon}\left( \int_{S^c}  |x|^2 d\mu_t^k dt + \int_{S^c} |x|^2 d\mu_t dt \right)
\end{aligned}
\end{equation*}
where we used Young's inequality with the constant $\epsilon$. By (\ref{h01}) we have
\begin{equation*}
\lim_{k \to \infty} A \leqslant \frac{\epsilon M}{2} + \frac{\epsilon}{2}. 
\end{equation*}
Since $\epsilon$ is arbitrary we have $\lim_{k\to\infty} A =0$.\\

We now study $B$. Consider the measure $\gamma_t^k$ on $\R^m \times \R^m$ given by
\[\gamma_t^k=(Id\times T_{\mu_t^k}^{\nu})_{\#}\mu_t^k .\]
Recall that the measure $\gamma_t^k$ is the optimal plan with marginals $\mu_t^k$ and $\nu$. Since $\mu_t^k \rightarrow \mu_t$, by the stability of optimal plans \cite[Theorem 5.20]{Villani}, the set of optimal plans between $\mu_t^k$ and $\nu$ is compact in the narrow topology and every limit point is an optimal plan between $\mu_t$ and $\nu$. On the other hand, because $\mu_t$ is an absolutely continuous measure, Brenier-McCann Theorem ensures that the optimal plan between $\mu_t$ and $\nu$ is unique. This implies that the sequence $\gamma_t^k$ converges narrowly for all $t \in [0,\hat{t}]$. Furthermore, the uniform convergence of $\mu_t^k$ implies that $\gamma_t^k$ have uniformly bounded second moment. We have 
\begin{equation}\label{y01}
\gamma_t^k=(Id\times T_{\mu_t^k}^{\nu}))_{\#}\mu_t^k \xrightarrow []{narrow} \gamma_t=(Id\times T_{\mu_t}^{\nu}))_{\#}\mu_t ~~~~ \forall t \in [0,\hat{t}].
\end{equation}
Taking the limit of $B$ yields
\begin{equation*}
\begin{aligned}
\lim_{k \to \infty}B &= \lim_{k \to \infty}\left| \int_0^{\hat{t}} \int_{\mathbb{R}^m} \left\langle V_t^k , T_{\mu_t^k}^{\nu}-T_t \right\rangle d\mu_t^k dt \right|\\
&\leqslant \lim_{k \to \infty} \frac{\epsilon}{2} \int_0^{\hat{t}} \int_{\mathbb{R}^m} | V_t^k |^2  d\mu_t^k dt + \frac{1}{2\epsilon} \int_0^{\hat{t}} \int_{\mathbb{R}^m} | T_{\mu_t^k}^{\nu}-T_t |  d\mu_t^k dt\\
&\leqslant \frac{\epsilon M}{2} + \frac{1}{2\epsilon} \lim_{k \to \infty} \int_0^{\hat{t}} \int_{\mathbb{R}^m} | T_{\mu_t^k}^{\nu}-T_t |^2  d\mu_t^k dt.
\end{aligned}
\end{equation*}
by lifting to the optimal plans $\gamma_t^k=(Id\times T_{\mu_t^k}^{\nu})_{\#}\mu_t^k$, we have
\begin{equation*}
\lim_{k \to \infty}B \leqslant \frac{\epsilon M}{2} + \frac{1}{2\epsilon} \lim_{k \to \infty} \int_0^{\hat{t}} \int_{\mathbb{R}^m \times \mathbb{R}^m} |y-T_t(x)|^2  d\gamma_t^k  dt.
\end{equation*}
Since $\gamma_t^k \to \gamma_t$ point-wise, $\gamma_t^k$ has uniformly bounded second moment, and $|y-T_t(x)|^2$ is dominated by a constant times $|x^2+y^2+1|$, we can use dominated convergence theorem. Therefore 
\begin{equation*}
\begin{aligned}
\lim_{k \to \infty}B & \leqslant \frac{\epsilon M}{2} + \frac{1}{2\epsilon} \lim_{k \to \infty} \int_0^{\hat{t}} \int_{\mathbb{R}^m \times \mathbb{R}^m} |y-T_t(x)|^2  d\gamma_t^k  dt\\
& = \frac{\epsilon M}{2} + \frac{1}{2\epsilon} \int_0^{\hat{t}} \int_{\mathbb{R}^m \times \mathbb{R}^m} |y-T_t(x)|^2  d\gamma_t  dt\\
&=\frac{\epsilon M}{2} + \frac{1}{2\epsilon} \int_0^{\hat{t}} \int_{\mathbb{R}^m} | T_{\mu_t}^{\nu}-T_t |^2 d\mu_t  dt\\
&\leqslant \frac{\epsilon M}{2} + \frac{\epsilon}{2}.
\end{aligned}
\end{equation*}

Finally, we study the last term $C$. We have
\begin{equation*}
\begin{aligned}
\lim_{k \to \infty}C &= \lim_{k \to \infty}\left| \int_0^{\hat{t}} \int_{\mathbb{R}^m} \left\langle V_t^k , T_t \right\rangle  d\mu_t^k dt - \int_0^{\hat{t}} \int_{\mathbb{R}^m} \left\langle V_t , T_{\mu_t}^{\nu} \right\rangle  d\mu_t dt \right|\\
&\leqslant \lim_{k \to \infty} \left| \int_0^{\hat{t}} \int_{\mathbb{R}^m} \left\langle V_t^k , T_t \right\rangle  d\mu_t^k dt - \int_0^{\hat{t}} \int_{\mathbb{R}^m} \left\langle V_t , T_t \right\rangle  d\mu_t dt \right|\\
& + \left| \int_0^{\hat{t}} \int_{\mathbb{R}^m} \left\langle V_t , T_{\mu_t}^{\nu}-T_t \right\rangle  d\mu_t dt \right|.
\end{aligned}
\end{equation*}
Since $T_t \in C_b^0$ we can use weak convergence of $V_t^k$ for the first term. Hence
\begin{equation*}
\begin{aligned}
\lim_{k \to \infty}C & \leqslant \left| \int_0^{\hat{t}} \int_{\mathbb{R}^m} \left\langle V_t , T_{\mu_t}^{\nu}-T_t \right\rangle  d\mu_t dt \right|\\
&\leqslant M\ \int_0^{\hat{t}} \int_{\mathbb{R}^m} |T_{\mu_t}^{\nu}-T_t|^2 d\mu_t dt\\
&\leqslant M \epsilon^2.
\end{aligned}
\end{equation*}
\end{proof}


\begin{theorem}[Existence and uniqueness of the flow]\label{theorem1}
Let $E:\mathcal{P}_2(\mathbb{R}^m) \longrightarrow [0,+\infty]$ be a lower semi continuous energy functional with locally compact sub-level sets and let $D(E)\subseteq \mathcal{P}_2^a(\mathbb{R}^m)$. Assume $E(\mu)< c < +\infty$ and that $E$ is restricted $\lambda$-convex at $\mu$. Then there exist $\hat{t}>0$ and a curve $\mu_t \in AC^2 \left( [0,\hat{t}]; \mathcal{P}_2^a(\mathbb{R}^m) \right)$ such that $\mu_t$ is the unique gradient flow of $E$ starting from $\mu$.
\end{theorem}

\begin{proof}
Let $\mu_t^k:=\mu_t^{\tau}$ be a piecewise constant solution to the minimizing movement scheme (\ref{GMM}) with $\tau=\frac{1}{k}$. The minimizing movement sequence is designed in a way that it converges to a limiting curve in a very general setting. In \cite[Theorem 11.1.6]{AmbrosioBook} it has been proved that, under very weak assumptions which hold here, the minimizing movement scheme converges sub-sequentially to a limiting curve such that (after relabelling) $\forall a>0$
\begin{itemize}
\item $\mu_t^k\xrightarrow []{W_2} \mu_t \in AC^2 \left( [0,a]; \mathcal{P}_2(\R^m) \right)$ uniformly in $[0,a]$.
\item The sequence $\lbrace V_t^k\rbrace$ of the velocity tangent vectors to $\lbrace \mu_t^k \rbrace$ converges weakly to $V_t \in L^2(d\mu)$ in $\R^m \times (0,T)$.
\item The continuity equation $\partial_t \mu_t + \nabla . \left( \mu_t V_t \right)=0$ holds for the limiting curve.
\end{itemize}
We need to prove that the limiting curve $\mu_t$ satisfies the steepest descent equation and that it is unique. Since $\mu_t$ is a continuous curve, we can find $\hat{t}$ such that $\mu_t \in B_{\delta/4}(\mu)$ for all $t \in [0,\hat{t}]$ where $\delta$ is the radius of restricted $\lambda$-convexity at $\mu$. We have
\begin{equation}\label{A}
E(\mu_t)\leqslant \liminf_{k \to \infty} E(\mu_t^k) \leqslant E(\mu) < c. 
\end{equation}
The first inequality follows from lower semi continuity of the energy and the second inequality follows from the structure of the minimizing movement scheme (\ref{GMM}). Hence 
\begin{equation}\label{B}
\mu_t \in B_{\delta/4}(\mu) \cap E_c ~~~ \forall t \in [0,\hat{t}].
\end{equation}
Since $\mu_t^k\xrightarrow []{W_2} \mu_t$ uniformly in $[0,\hat{t}]$, we can find $K\in \mathbb{N}$ such that $W_2(\mu_t,\mu_t^k)<\delta/4$, $\forall k\geqslant K$ and $\forall t \in [0,\hat{t}]$. Without loss of generality we assume that $K=1$. Therefore (\ref{A}) and (\ref{B}) imply
\begin{equation}\label{D}
\mu_t^k\in B_{\delta/2}(\mu) \cap E_c.
\end{equation}
The Euler-Lagrange equation of minimizing movement (\ref{sub of piecwise}) implies that $-V_t^k\in \partial E(\mu_t^k)$. Therefore using (\ref{D}) and the variational formulation of the subdifferential (Lemma \ref{lem:subdif}) we have
\begin{equation}\label{ine:Mk}
E(\nu)-E(\mu_t^k) \geqslant \int_{\R^m} \left\langle -V_t^k,T_{\mu_t^k}^{\nu} -Id \right\rangle d\mu_t^k + \dfrac{\lambda}{2} W_2^2(\mu_t^k,\nu)
\end{equation}
for all $\nu \in  B_{\delta/2}(\mu) \cap E_c$. By construction, $B_{\delta/4}(\mu_t) \cap E_c \subseteq B_{\delta/2}(\mu_t^k) \cap E_c$. Therefore (\ref{ine:Mk}) holds for all $\nu$ in $B_{\delta/4}(\mu_t) \cap E_c$. By integrating (\ref{ine:Mk}) over $t$ and against a text function $\psi \in C_c^{\infty}((0,\hat{t});[0,\infty))$ we have
\begin{small}
\begin{equation}\label{m01}
\int_0^{\hat{t}}E(\nu) \psi(t) dt -\int_0^{\hat{t}}E(\mu_t^k) \psi(t) dt \geqslant \int_0^{\hat{t}}\int_{\R^m}\left\langle -V_t^k,T_{\mu_t^k}^{\nu} -Id \right\rangle \psi(t) d\mu_t^k dt + \dfrac{\lambda}{2} \int_0^{\hat{t}} W_2^2(\mu_t^k,\nu)\psi(t)dt
\end{equation}
\end{small}
for all $\nu$ in $B_{\delta/4}(\mu_t) \cap E_c$.

We take the limit of (\ref{m01}) as $k \to \infty$. By the lower semi-continuity of $E$
\begin{equation*}
\int_0^{\hat{t}}E(\nu)\psi(t)dt -\int_0^{\hat{t}}E(\mu_t)\psi(t)dt \geqslant \int_0^{\hat{t}}E(\nu)\psi(t)dt - \liminf_{k\to\infty}\int_0^{\hat{t}}E(\mu_t^k)\psi(t)dt.
\end{equation*}
Lemma \ref{lemma:convergence} implies that 
\begin{equation*}
\lim_{k \to \infty} \int_0^{\hat{t}} \int_{\R^m} \left\langle V_t^k,T_{\mu_t^k}^{\nu}-Id \right\rangle \psi(t) d\mu_t^k dt = \int_0^{\hat{t}} \int_{\R^m} \left\langle V_t,T_{\mu_t}^{\nu}-Id \right\rangle \psi(t) d\mu_t dt.
\end{equation*}
By the triangle inequality
\begin{equation}\label{m02}
W_2(\mu_t,\nu)-W_2(\mu_t^k,\mu_t) \leqslant W_2(\mu_t^k,\nu) \leqslant W_2(\mu_t,\nu)+W_2(\mu_t^k,\mu_t).
\end{equation}
Therefore $\lim_{k\to\infty}W_2(\mu_t^k,\nu)=W_2(\mu_t,\nu)$. Furthermore, since $\mu_t^k \xrightarrow[]{W_2}\mu_t$ uniformly, inequality (\ref{m02}) implies that $W_2(\mu_t^k,\nu)$ is uniformly bounded. Hence, by dominated convergence theorem
\begin{equation*}
\lim_{k\to\infty}\int_0^{\hat{t}} W_2^2(\mu_t^k,\nu)\psi(t)dt=\int_0^{\hat{t}} W_2^2(\mu_t,\nu)\psi(t)dt.
\end{equation*}
In conclusion $\forall \nu \in B_{\delta/4}(\mu_t) \cap E_c$ and $\forall \psi \in C_c^{\infty}((0,\hat{t});[0,\infty))$ we have
\begin{small}
\begin{equation}\label{m03}
\int_0^{\hat{t}}E(\nu)\psi(t)dt -\int_0^{\hat{t}}E(\mu_t)\psi(t)dt \geqslant \int_0^{\hat{t}} \int_{\R^m}\left\langle -V_t,T_{\mu_t}^{\nu} -Id \right\rangle \psi(t) d\mu_t dt + \dfrac{\lambda}{2} \int_0^{\hat{t}} W_2^2(\mu_t,\nu)\psi(t)dt.
\end{equation}
\end{small}
Let $t_0$ be a Lebesgue point of the map $t \mapsto \int_0^{\hat{t}}E(\mu_t)\psi(t)dt + \int_{\R^m}\left\langle V_t,T_{\mu_t}^{\nu} -Id \right\rangle d\mu_t + \frac{\lambda}{2} \int_0^{\hat{t}} W_2^2(\mu_t,\nu)\psi(t)dt$. By considering a sequence of smooth mollifiers $\psi_n$ converging to the delta function at $t_0$, the inequality (\ref{m03}) is reduced to
\begin{equation}\label{m04}
E(\nu) -E(\mu_{t_0}) \geqslant \int_{\R^m}\left\langle -V_{t_0},T_{\mu_{t_0}}^{\nu} -Id \right\rangle d\mu_{t_0}+ \dfrac{\lambda}{2} W_2^2(\mu_{t_0},\nu).
\end{equation}
Therefore (\ref{m04}) holds for almost all $t$. By Lemma \ref{lem:subdif} we have
\begin{equation*}
V_t \in -\partial E(\mu_t) ~~~\text{for almost all $t\in[0,\hat{t}]$}.
\end{equation*}

We now study uniqueness of the solution. The available uniqueness proofs in the case of $\lambda$-convexity can be repeated in the domain of restricted $\lambda$-convexity because as soon as the flow exists clearly it dissipates the energy and the trajectories of the flow starting from $\mu$ remain in the domain of restricted $\lambda$-convexity for a short time where one can use $\lambda$-convexity. Hence uniqueness arguments can be repeated similar to the available proofs such as in \cite[Theorem 11.1.4]{AmbrosioBook}. Therefore we provide only the key ideas here.\\

Assume that we have two gradient flows $\mu_t^1$ and $\mu_t^2$ both starting from $\mu$. One can show that $W_2^2(\mu_t^1,\mu_t^2)$ is absolutely continuous in time and 
\begin{equation*}
\frac{d}{dt}W_2^2(\mu_t^1,\mu_t^2) \leqslant \lim_{h\downarrow 0}\frac{W_2^2(\mu_{t+h}^1,\mu_{t+h}^2)-W_2^2(\mu_{t+h}^1,\mu_t^2)}{h} + \lim_{h\downarrow 0}\frac{W_2^2(\mu_{t+h}^1,\mu_{t}^2)-W_2^2(\mu_t^1,\mu_t^2)}{h}.
\end{equation*}
By differentiability of Wasserstein metric (\ref{derivative of metric}), for almost all $t \in [0,\hat{t}]$ we have
\begin{equation*}
\frac{1}{2} \lim_{h \to 0}\frac{W_2^2(\mu_{t+h}^1,\mu_{t}^2)-W_2^2(\mu_t^1,\mu_t^2)}{h}=\int_{\R^m} \left\langle V_t^1, Id-T_{\mu_t^1}^{\mu_t^2} \right\rangle d\mu_t^1.
\end{equation*}
Therefore
\begin{equation}\label{l01}
\frac{d}{dt}W_2^2(\mu_t^1,\mu_t^2) \leqslant \int_{\R^m} \left\langle V_t^2, Id-T_{\mu_t^2}^{\mu_t^1} \right\rangle d\mu_t^2 + \int_{\R^m} \left\langle V_t^1, Id-T_{\mu_t^1}^{\mu_t^2} \right\rangle d\mu_t^1.
\end{equation}
Consider (\ref{m04}) along $\mu_t^1$. We have
\begin{equation}\label{l02}
E(\mu_t^2) -E(\mu_t^1) \geqslant \int_{\R^m}\left\langle -V_t,T_{\mu_t^1}^{\mu_t^2} -Id \right\rangle d\mu_t^1+ \dfrac{\lambda}{2} W_2^2(\mu_t^1,\mu_t^2).
\end{equation}
Rewriting (\ref{l02}) again along $\mu_t^2$ and using (\ref{l01}) result in
\begin{equation*}
\dfrac{1}{2}\dfrac{d}{dt}W_2^2(\mu_t^2,\mu_t^1)\leqslant -\lambda W_2^2(\mu_t^2,\mu_t^1).
\end{equation*}
Hence
\begin{equation*}
W_2(\mu_t^2,\mu_t^1)\leqslant e^{-\lambda t}W_2(\mu_0^2,\mu_0^1)=0 ~~~~\forall t \in [0,\hat{t}].
\end{equation*} 
\end{proof}\\


\section{Wasserstein gradient flow of the Dirichlet energy}

In this section we prove a local well-posedness result for the gradient flow of the Dirichlet energy on $S^1$. Energy or in short $E$ in this section always refers to the \textbf{Dirichlet energy}
\begin{equation}\label{Dirichlet energy}
E(\mu)=\left\lbrace \begin{aligned} &\dfrac{1}{2}(\partial_x u)^2 \quad &\text{if } \mu=udx, u \in H^1(S^1), \\ &+\infty &\text{else.}~~~~~  \end{aligned} \right. 
\end{equation} 
When convenient, we refer to an absolutely continuous measure $\mu=udx$ by its density $u$. In particular, by a smooth or positive measure we mean a measure with a smooth or positive density.


\subsection{Dirichlet energy on $S^1$}

The underlying space of the measures that we study in this section is $S^1$. We identify $S^1$ with $\mathbb{R}/\mathbb{Z}$. Because $S^1$ is a manifold, the theory developed in the previous section should be slightly modified. On a Riemannian manifold, there is the issue of existence and regularity of the optimal maps. This question has been an active area of research. Ma-Trudinger-Wang condition in \cite{MTW} is a famous example that studies this issue. In the case of $S^m$, the problem has been addressed and positive results are available such as \cite{McCannKim} and \cite{Loeper}. The results guarantee that between any pair of smooth positive measures $\mu,\nu$ on $S^m$ there exists a unique smooth optimal map $T_{\mu}^{\nu}$. To apply Theorem \ref{theorem1} to $S^1$, one has to replace the inner product in the subdifferential definition (\ref{def:subdiff}) with the inner product on the tangent space of $S^1$. For a given optimal map $\hat{T}$, the distance $|\hat{T}(x)-x|$ might not coincide with the geodesic distance of $S^1$. As was suggested in \cite{Slepcev}, this problem can be solved by representing $T:[0,1]\rightarrow[-\frac{1}{2},\frac{3}{2}]$ where $T(x)$ is the smallest element of $\hat{T}(x) \in \R/\Z$ such that $ |T(x)-x| \leqslant \frac{1}{2}$ and by relabelling $S^1$ by $[T(0),T(0)+1]$. It is easy to check that $T$ is an optimal map from $[0,1]$ to $[T(0),T(0)+1]$, furthermore $T$ is monotone and the geodesic distance of $S^1$ coincides with $|T(x)-x|$.\\

Let $\mu_0=u_0dx$ and $\mu_1=u_1dx$ be two smooth measures on $S^1$ and let $T$ be the optimal map between them. Then by Monge-Amp\`ere equation (\ref{Monge}) we have
\begin{equation*}
u_1(T(X))=\dfrac{u_0(x)}{T'(x)}.
\end{equation*}
Since the geodesics are given by the push forward of linear interpolation of the optimal map and the identity map, the explicit form of the geodesic $u_s$ between $u_0$ and $u_1$ is given by
\begin{equation*}
u_s((1-s)x+sT(x))=\dfrac{u_0(x)}{(1-s)+sT'(x)}.
\end{equation*}
In the notation of the previous section, if we think of $f=T-Id$ as the tangent vector field that connects $u_0$ to $u_1$, the geodesic equation can be written as
\begin{equation}\label{interpolant equ}
u_s(x+sf(x))=\dfrac{u_0(x)}{1+sf'(x)}.
\end{equation}

We will see in Lemma \ref{lem smooth app.} that for studying  restricted $\lambda$-convexity of the energy, it is enough to consider measures with smooth and positive densities. By the derivative formulation of $\lambda$-convexity (\ref{def:smooth conv along curve}) the energy is $\lambda$-convex along the geodesic $\mu_s$ connecting $\mu_0$ to $\mu_1$ if
\begin{equation}\label{smooth convexity}
\dfrac{d^2}{ds^2}E(u_s)\geqslant \lambda W_2^2(u_0,u_1).
\end{equation}

We start by proving that the Dirichlet energy is not $\lambda$-convex for any $\lambda$. This is known to the community, and in \cite{Slepcev} Carrillo and Slepčev proved that the Dirichlet energy is not convex on $S^1$. We will study a scalable family of functions to prove the lack of convexity for any $\lambda \in \mathbb{R}$. Let $u_s$ be the geodesic connecting $u_0=u$ to $u_1$, and let $f$ be the corresponding tangent vector field. We compute the second derivative of the energy along the geodesic
\begin{align*}
\left. \dfrac{d^2 E(u_s)}{ds^2}\right|_{s=0}
=&\left. 
\frac{d^2}{ds^2}
\right|_{s=0}
\int_{S^1} (\partial_y u_s(y))^2 \, dy\,.
\end{align*} 
By the Monge-Amp\`ere equation (\ref{interpolant equ}) and the change of variables $y=x+sf(x)$ we have

\begin{equation}\label{SecondDerivative}
\begin{aligned}
\left. \frac{d^2 E(u_s)}{ds^2}\right|_{s=0} 
&= \left. \frac{d^2}{ds^2}\right|_{s=0} \int_{S^1} \left(\frac{\partial x}{\partial y} \frac{\partial}{\partial x}
\frac{u(x)}{1+sf'(x)}\right)^2 \frac{\partial y}{\partial x} dx\\
&= \left. \frac{d^2}{ds^2}\right|_{s=0} \int_{S^1} \left(\frac{1}{1+sf'(x)}\partial_x 
\frac{u(x)}{1+sf'(x)}\right)^2 (1+sf'(x)) \, dx\\
&=\left. \frac{d^2}{ds^2}\right|_{s=0} \int_{S^1} \frac{((1+sf'(x))u'(x)-su(x)f''(x))^2}{(1+sf'(x))^3}  \, dx \\
&= 2\int_{S^1}  \bigl\{ (f'' u)^2 + 8 (f'' u) (f' u') +  6 (f' u')^2 \bigr\} \, dx.
\end{aligned}
\end{equation} 
If the energy $E$ is $\lambda$-convex, then
\begin{equation}\label{equ:cal}
\underbrace{\int_{S^1}
\bigl\{ (f'' u)^2 + 8(f'' u)(f' u') + 6 (f'u')^2\bigr\} \,dx}_A \geqslant
\lambda \underbrace{ \int_{S^1} f^2 u\, dx}_B. 
\end{equation}
We use (\ref{equ:cal}) as a guide to find a counter-example. In the example, $u$ and $f$ are not smooth, and we cannot directly apply this computation but, Lemma \ref{lem smooth app.} validates the calculations.\\

We view $S^1$ as the interval $[-1/2,1/2]$ with the endpoints identified. The construction of the example is simple: let $u=1-4|x|$ and $f'=u^{-1}$, this forces the integrand of $A$ to be negative, and the rest follows from a scaling argument. We have to make some modifications to the functions so that the integral converges and the mass is normalized to $1$. We define $u$ and $f'$ as follows
\begin{equation*}
  u(x)= \begin{cases}
         \frac{81}{16}(1-4|x|)\,,  \quad &0\leqslant |x|\leqslant \frac{2}{9} , \\
          \frac{9}{16} \,,      &\frac{2}{9} \leqslant |x|\leqslant \frac{3}{8}, \\
          \frac{9}{4}(1-2|x|) \,,  &\frac{3}{8}\leqslant |x| \leqslant \frac{1}{2}         
        \end{cases}
          \quad f'(x)= \begin{cases}
          \frac{16}{81(1-4|x|)} \,,         \quad  &0\leqslant |x|\leqslant \frac{2}{9}  , \\
		  \frac{16 }{11}(3-8|x|)\,, 		         &\frac{2}{9}  \leqslant |x|\leqslant \frac{3}{8}, \\
		  0 \,,        			 		&\frac{3}{8}\leqslant |x| \leqslant \frac{1}{2}.
  \end{cases}
\end{equation*}

\begin{figure}[htb]
\centering
\includegraphics[width=0.5 \textwidth]{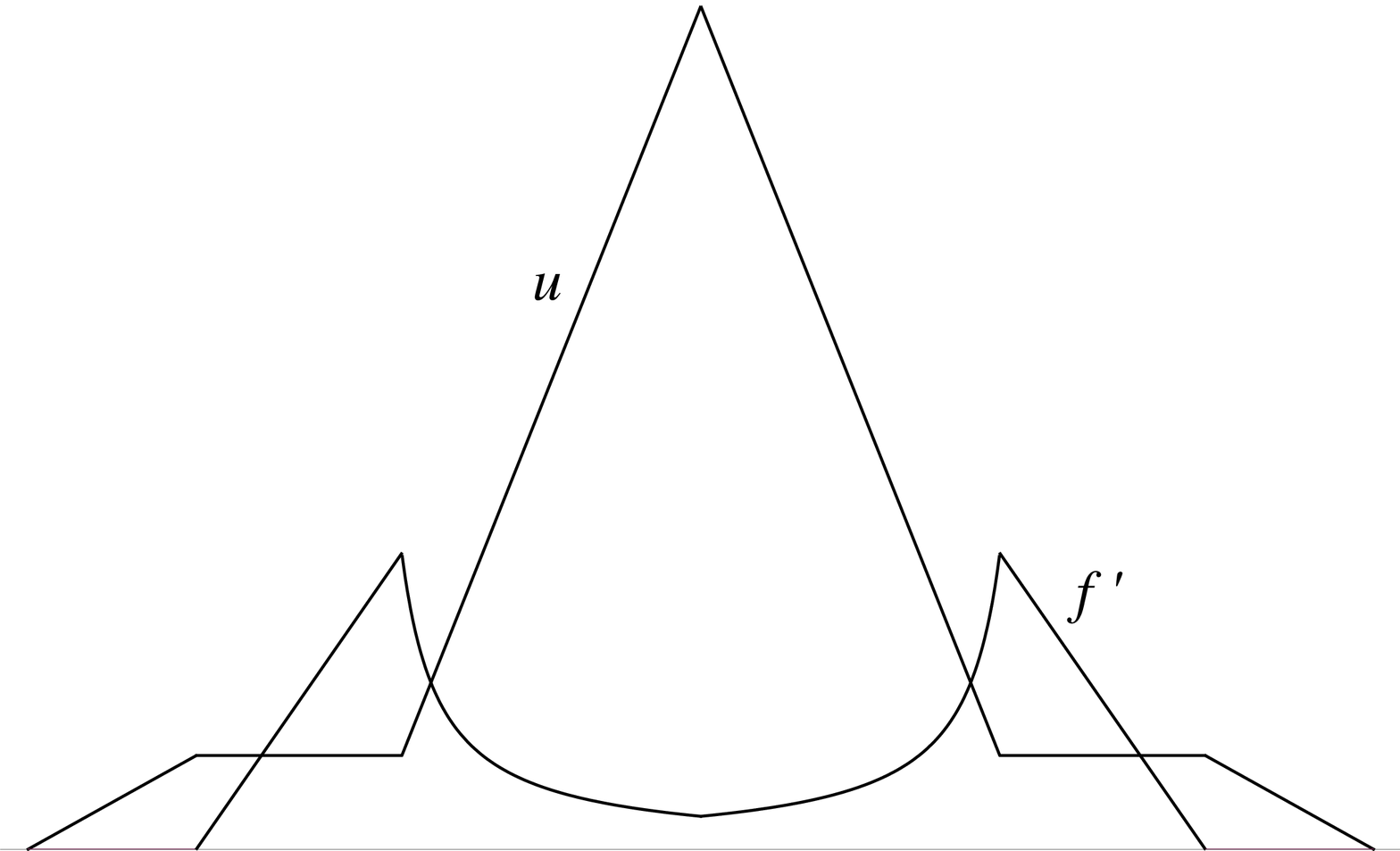}
\caption{Graph of $u$ and $f'$}
\label{fig:fandu}
\end{figure}

By scaling $u_h(x):=hu(h x)$ and $f'_h(x):=\frac{1}{h} f'(h x)$ we have
\begin{equation*}
A=-C_1 h^2\,, 
\quad B=C_2
\end{equation*}
for some positive constants $C_1,C_2$. If the energy is $\lambda$-convex then we must have $A \geqslant \lambda B$ and it should hold uniformly for any such $u$ and $f$. But for a fixed $\lambda$, we can choose $h$ large enough so that the opposite inequality holds. This means that the Dirichlet energy is not $\lambda$-convex on $\mathcal{P}_2(S^1)$.\\

In the example above, by pushing $h$ to larger numbers the lack of convexity becomes worse. By looking at the equation of $u(x)$, it is clear that the Dirichlet energy of $u(x)$ gets bigger for larger values of $h$. This example hints that one of the obstructions against the $\lambda$-convexity of the Dirichlet energy is the magnitude of the energy which can be controlled on energy sub-level sets.\\


\subsection{Restricted $\lambda$-convexity of Dirichlet energy}

\begin{lemma}[Uniform convergence on energy sub-level sets.]\label{theo:Uniform con.}
Uniform convergence and Wasserstein convergence are equivalent on energy sub-level sets of Dirichlet energy on $S^1$. In particular for two measures $\mu_1=u_1dx$ and $\mu_2=u_2dx$ with $E(\mu_1),E(\mu_2)< c <+\infty$ we have
\begin{equation}\label{equ:equivalence}
W_2^2(\mu_1,\mu_2)\geqslant \alpha |u_1-u_2|_{\infty}^{\beta}
\end{equation}
where $\alpha=\alpha(c)$ and $\beta$ are constants.
\end{lemma}

\begin{proof}
One side of the equivalence is easy. Assuming $u_n \xrightarrow[]{uniform}u_0 $ we have
\begin{equation*}
\int_{S^1}\psi u_ndx \longrightarrow \int_{S^1}\psi udx ~~~\forall \psi \in C^0(S^1)
\end{equation*}
which implies Wasserstein convergence of $\mu_n=u_ndx$ to $\mu=udx$ by (\ref{narrow}) and finiteness of the second moments on $S^1$.\\

For the converse inequality, we first study the regularity of a measure with finite energy. Let $\nu=vdx \in E_c$. By Poincare's inequality and $\int_{S^1}vdx=1$ we have
\begin{equation*}
\int_{S^1}|v|^2dx \leqslant \int_{S^1}|v'|^2dx + 2\int_{S^1}|v|dx + \int_{S^1}dx \leqslant c + 3.
\end{equation*}
Therefore $H^1$-norm of $v$ is bounded by its energy. The Sobolev embedding theorem implies that $v$ is $C^{0,1/2}$ continuous and we have
\begin{equation}\label{module of continuity}
\left| v(x)-v(y) \right| \leqslant \left( \int_{S^1} |v'|^2dx\right)^{1/2} \left( \int_x^y dx \right)^{1/2} \leqslant \sqrt{c|x-y|} ~~~\forall x,y\in S^1.
\end{equation}
Therefore the modulus of continuity is $\sqrt{c}$. Let $\mu_1$ and $\mu_2$ be as in the assumption. Therefore $u_1$ and $u_2$ are $C^{0,1/2}$ continuous with constant $\sqrt{c}$. Assume that $|u_1-u_2|_{\infty}\geqslant h>0$. In particular without loss of generality assume that for some point $x_0 \in S^1$ we have $u_1(x_0)-u_2(x_0)\geqslant h$. For every $x\in S^1$, we have 
\begin{equation}
\begin{aligned}
&u_1(x)\geqslant u_1(x_0)-\sqrt{c|x|}\\
&u_2(x)\leqslant u_2(x_0)+\sqrt{c|x|}.
\end{aligned}
\end{equation}

\begin{figure}
\centering
\includegraphics[width=0.6\textwidth]{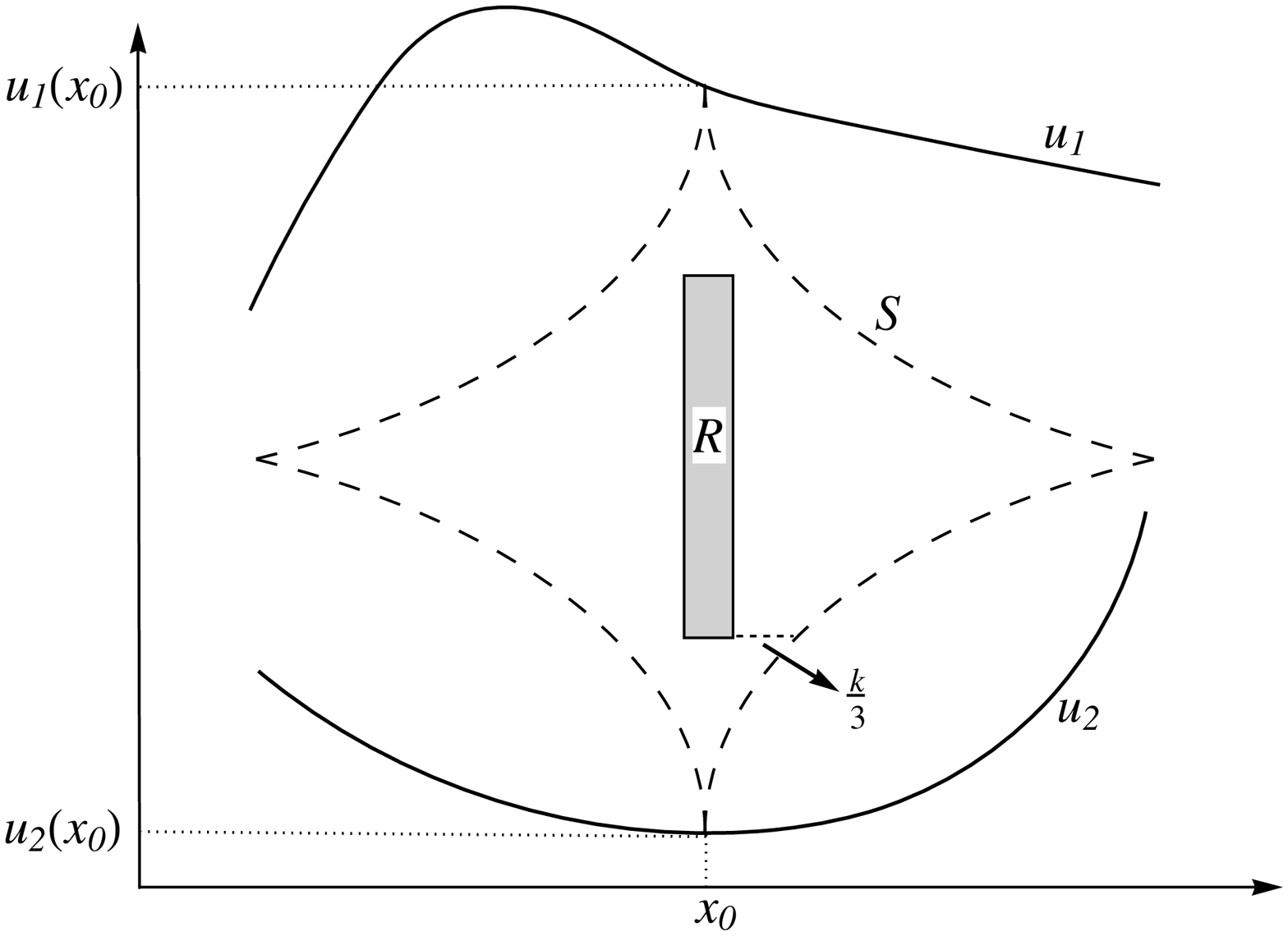}
\caption{Wasserstein $\Longleftrightarrow$ Uniform}
\label{fig:star}
\end{figure}

Therefore $u_1$ lies above and $u_2$ is below the star-like shape in Figure (\ref{fig:star}). Call the star-like shape by S. Consider a rectangle R in the center of S with height $\frac{h}{2}$ and width $\frac{k}{3}$ where $k$ is the width of S at the height $u_2(x_0)+\frac{h}{4}$. We have $k=\frac{ h^2}{8c}$ and the area of R is given by $\frac{h^3}{48c}$. In order to transport the measure $\mu_1$ to $\mu_2$, some mass at least equal to the area of R should be transported outside of S. Therefore
\begin{equation*}
\begin{aligned}
W_2^2(\mu_1,\mu_2) &\geqslant \lbrace \text{area of R} \rbrace \lbrace \text{distance required to move R outside of S} \rbrace^2 \\
&\geqslant \frac{h^3}{48c}.(\frac{k}{3})^2 \\
&\geqslant \dfrac{1}{384 c^3}|u_1-u_2|_{\infty}^7.
\end{aligned}
\end{equation*}
\end{proof}


\begin{lemma} [Lower Semi Continuity.]\label{lem:l.s.c}
The Dirichlet energy is lower semi continuous with respect to the Wasserstein metric on $S^1$.
\end{lemma}
\begin{proof}
Let $u_n \xrightarrow[]{W_2} u$. Since we have to prove $E(u)\leqslant \liminf_{n \to \infty} E(u_n)$, we can assume that $E(u_n)$ is bounded. This implies that $H^1$-norm of the sequence is bounded. By Banach–Alaoglu theorem, $u_n$ has a weak limit point $v \in H^1$. Therefore $u_n$ converges to $v$ strongly in $L^2$. We claim that $v=u$. Let $\psi \in C^0 (S^1) $ we have
\begin{align*}
\left| \int_{s^1} \psi(x) (u-v)(x)dx \right|\leqslant & \left| \int_{s^1} \psi(x) (u-u_n)(x)dx \right| + \left| \int_{s^1} \psi(x) (u_n-v)(x)dx \right| \\ 
\leqslant & \left| \int_{s^1} \psi(x) (u-u_n)(x)dx \right| +  (\int_{s^1} |\psi|^{2}dx)^{1/2} (\int_{s^1} |u_n-v|^2 dx)^{1/2}
\end{align*}
By Lemma \ref{theo:Uniform con.} the Wasserstein and uniform convergences are equivalent on energy sub-level sets. Therefore the first term in the last inequality goes to zero. The second term also goes to zero because $u_n$ converges to $v$ strongly in $L^2$. Hence $u=v$ almost everywhere. Because $u$ and $v$ are continuous, we have $u=v$. The Dirichlet energy is known to be lower semi continuous under weak $H^1$ convergence (for example see \cite[Theorem 8.2.1]{Evans}). Hence we have $E(u)=E(v)\leqslant \liminf_{n \to \infty} E(u_n)$.
\end{proof}\\


The following lemma validates smooth calculation in the sense that for studying restricted $\lambda$-convexity of the energy, one can study restricted $\lambda$-convexity of the energy only on smooth measures.\\

\begin{lemma}[Approximation by smooth measures.]\label{lem smooth app.}
Let $\mu \in E_c$. Assume that the energy is restricted $\lambda$-convex on smooth measures in $B_{\delta}\cap E_{c}$. Then $E$ is restricted $\lambda$-convex at $\mu$.
\end{lemma}
\begin{proof}
Let $\mu_0=u_0dx,~\mu_1=u_1dx \in B_{\delta}\cap E_c$ and let $\eta_k$ be a standard smooth mollifier converging to the Dirac delta function. Define $u_{k,i}(x):= \eta_k \ast u_i(x)$ for $i=0,1$ where $\ast$ is the convolution on $S^1$. Since $ u_{k,i} \xrightarrow{uniformly} u_i$, by Lemma \ref{theo:Uniform con.} we have $u_{k,i} \xrightarrow{W_2} u_i$. Therefore for large enough $k$ we have $u_{k,i} \in B_{\delta}(u)$. The energy also converges, because
\begin{equation}\label{x01}
E(u_{k,i})=\int_{S^1} \left( \partial_x (u_i \ast \eta_k) \right) ^2 dx=\int_{S^1}\left( (\partial_x u_i) \ast \eta_k \right) ^2 dx \xrightarrow[k \to \infty]{} \int_{S^1}\left( \partial_x u_i\right)^2 dx = E(u_i).
\end{equation}
Hence $u_{k,i} \in B_{\delta}(u) \cap E_c$ for large enough $k$. By smoothness of $u_{k,i}$ and the assumption of the lemma, we have $\lambda$-convexity of the energy along the geodesics $u_{k,s}$ connecting $u_{k,0}$ to $u_{k,1}$
\begin{equation}\label{DDD}
E(u_{k,s})\leqslant (1-s)E(u_{k,0})+sE(u_{k,1})-\dfrac{\lambda}{2}s(1-s)W_2^2(u_{n,0},u_{k,1}).
\end{equation}
Let $\gamma_k$ and $\gamma$ be in order the optimal plan connecting ${\mu_{k,0}=u_{k,0}dx}$ to $\mu_{k,1}=u_{k,1}dx$ and the optimal plan connecting $\mu_0$ to $\mu_1$. By stability of the optimal plans \cite[Theorem 5.20]{Villani} $\gamma_k$ converges in narrow topology to $\gamma$ along a subsequence which after relabelling we assume to be the whole sequence. Equivalence of narrow and Wasserstein convergence (\ref{equ:equivalence}) on $S^1 \times S^1$ implies
\begin{equation*}
\mu_{k,s}=\left((1-s)\Pi^1 + s\Pi^2 \right)_{\#}\gamma_k  \xrightarrow{W_2} \left((1-s)\Pi^1+ s\Pi^2 \right) _{\#}\gamma =\mu_s
\end{equation*}
where $\Pi^i$ is the projection to the $i^{th}$ coordinate and $\mu_s$ is the geodesic connecting $\mu_0$ to $\mu_1$. The lower semi-continuity of Dirichlet energy \ref{lem:l.s.c} yields $E(\mu_s)\leqslant \liminf_{k \to +\infty}E(\mu_{k,s})$. Hence by taking the limit of (\ref{DDD}) we have
\begin{equation*}
E(\mu_s)\leqslant (1-s)E(\mu_0)+sE(\mu_1)-\dfrac{\lambda}{2}s(1-s)W_2^2(\mu_0,\mu_1).
\end{equation*}
\end{proof}\\


In the following lemma we prove that the energy is finite along a geodesic, provided that the energies of the end points are finite.\\

\begin{lemma}[Energy of the interpolant.]\label{lem interpolant regularity}
Let $\mu_0=u_0dx$ and $\mu_1=udx$ be two smooth measures with $E(\mu_0),E(\mu_1)< c< +\infty$ and $u_0,u_1 > m >0$. Then there are constants $\hat{c}<+\infty$ and $\hat{m}>0$ depending only on $c$ and $m$ such that $E(\mu_s) < \hat{c}$ and $u_s > \hat{m}$ along the geodesic $\mu_s$ connecting $\mu_0$ to $\mu_1$.
\end{lemma}

\begin{proof}
By $C^{0,1/2}$ continuity of the densities, there exists $M=M(c)$ such that $u_0(x),u_1(x)<M$ for all $x\in S^1$. Let $T:S^1 \longrightarrow S^1$ be the the optimal transport map between $u_0$ and $u_1$. By Monge–Ampère equation (\ref{Monge}) we have
\[ \left|  T'(x) \right| \leqslant \dfrac{M}{m} \]
By taking the derivative of Monge–Ampère equation (\ref{Monge}) we have
\begin{equation}\label{h001}
\begin{aligned}
|T''(x)|&=\left|\dfrac{u_0'(x)u_1(T(x))-u_0(x)T'(x)u_1'(T(x))}{u_1(x)^2}\right| \\
&\leqslant \dfrac{M}{m^2}|u'_0(x)| + \dfrac{M^2}{m^3} |u'_1(x)|. 
\end{aligned}
\end{equation}
Now let $\mu_s$ be the geodesic connecting $\mu_0$ to $\mu_1$. We have
\begin{equation}\label{equ:interpo}
u_s((1-s)x+sT(x))=\dfrac{u_0(x)}{(1-s)+sT'(x)}.
\end{equation}
Plugging in bounds on $T'$ yields
\begin{equation*}
u_s(x) \geqslant \dfrac{m^2}{M}.
\end{equation*}
Hence $u_s>\hat{m}$ where $\hat{m}=\hat{m}(c,m)$. Taking derivative of the equation (\ref{equ:interpo}) we have
\begin{equation*}
u'_s((1-s)x+sT(x))=\dfrac{u'_0(x)[(1-s)+sT'(x)]-su_0(x)T''(x)}{((1-s)+sT'(x))^3}
\end{equation*}
Using the bounds on $T'$ and (\ref{h001}) we have:
\begin{equation*}
\begin{aligned}
|u'_s((1-s)+sT'(x))| &\leqslant \dfrac{|u'_0(x)|}{|(1-s)+sT'(x)|^2} + \dfrac{|u_0(x)||T''(x)|}{|(1-s)+sT'(x)|^3}\\
&\leqslant (\dfrac{M}{m})^2 |u'_0(x)| + (\dfrac{M}{m})^5|u'_0(x)| + (\dfrac{M}{m})^6|u'_1(x)|.
\end{aligned}
\end{equation*}
Taking integral from both sides yields $E(u_s)<\hat{c}$ where $\hat{c}=\hat{c}(c,m)$.
\end{proof}\\


The idea of the following lemma was suggested by my supervisor Almut Burchard. This lemma will be used in calculations of the second derivative of the energy in Theorem \ref{theo:Uniform con.}. 
 
\begin{lemma}[Interpolation inequality.]\label{Maximum bound}
For every $\alpha \in \R$ there exists a constant $\lambda <0 $ such that
\[ |f''|_{L^2}  - \alpha |f'^2|_{\infty} - \lambda |f|_{L^2} \geqslant 0   ~~~ \forall f \in C^{\infty}(S^1).\]
\end{lemma}
\begin{proof}
Consider the Fourier expansion $f(x)=\sum_{k \in \mathbb{Z}} a_k e^{i 2k \pi x} $. We have
\begin{equation*}
|f'|_{\infty} = \sup_{x \in S^1} |\sum_{k \in \mathbb{Z}} i2k \pi a_k e^{i 2k \pi x} | \leqslant
2 \pi \sum_{k \in \mathbb{Z}} | k a_k | = 2\pi \sum_{k \in \mathbb{Z}} |k^4 a^2_k|^{\frac{4}{10}} |a^2_k|^{\frac{1}{10}}|\dfrac{1}{k}|^{\frac{6}{10}}.
\end{equation*}
Hölder's inequality with exponents $\frac{4}{10}$, $\frac{1}{10}$, and $\frac{5}{10}$ yields
\begin{equation*}
|f'|_{\infty} \leqslant 2\pi \left( \sum_{k \in \mathbb{Z}} (|k^4 a^2 _k| \right)^{\frac{4}{10}} \left( \sum_{k \in \mathbb{Z}}
|a^2 _k|\right)^{\frac{1}{10}} \left( \sum_{k \in \mathbb{Z}} |\dfrac{1}{k}|^{\frac{6}{5}}\right)^{\frac{1}{2}}.
\end{equation*}
The term $2\pi (\sum_{k \in \mathbb{Z}}|\dfrac{1}{k}|^{\frac{6}{5}})^{1/2}=d$ is a constant independent of $a_k$. Therefore
\begin{equation*}
|f'|_{\infty} \leqslant d |f''|_{L^2} ^{4/5} |f|_{L^2} ^{1/5}.
\end{equation*}
By the arithmetic-geometric inequality for a constant $\beta$ we have 
\begin{align*}
|f'|_{\infty} &\leqslant d |f''|_{L^2} ^{4/5} |f|_{L^2} ^{1/5}\\
&= d( {\beta}^{5/4}|f''|_{L^2})^{4/5} ({\beta}^{-5} |f|_{L^2})^{1/5}\\
&\leqslant \dfrac{4d}{5} {\beta}^{5/4} |f''|_{L^2} + \dfrac{d}{5} {\beta}^{-5} |f|_{L^2}.
\end{align*}
Putting $\beta=(\frac{5}{4}\alpha d)^{\frac{-4}{5}}$ and $\lambda = -\dfrac{d \alpha {\beta}^{-5}}{5}$ yields
\[ \alpha |f'^2|_{\infty} \leqslant |f''|_{L^2} - \lambda |f|_{L^2} .\]
\end{proof}\\


We are now ready to prove the main theorem of this section which shows that the Dirichlet energy is restricted $\lambda$-convex at positive measures.

\begin{theorem}[restricted $\lambda$-convexity of the Dirichlet energy.]\label{the:convexity}
Let $\mu=udx$ be a measure with $E(u)< c<+\infty$ and $u> m>0$. Then $\exists \lambda=\lambda_{c,m}$ such that $E$ is restricted $\lambda$-convex at $\mu$.
\end{theorem}

\begin{proof}
We first claim that the second derivative of the energy at a positive measure $\nu$ is uniformly bounded from below along any smooth vector field. Let $\nu=vdx$ be a measure with $E(v)<c$ and $v>m$. Let $\nu_1$ be another smooth measure and let $f$ be the vector field defining the geodesic  $\nu_s=(Id+sf)_{\#}\nu$ that connects $\nu$ to $\nu_1$. By (\ref{SecondDerivative}) we have
\begin{equation*}
\left. \dfrac{d^2 E(\nu_s)}{ds^2}\right|_{s=0} = 2\int_{S^1}  (f'' v)^2 + 8 (v f'') (v' f' ) +  6 (v' f' )^2 dx.
\end{equation*}
Recall that $W_2^2(\nu,\nu_1)=\int_{S^1}f(x)^2 v(x)dx$. By (\ref{smooth convexity}) the energy is $\lambda$-convex at $v$, if for all such vector fields 
\begin{equation*}
2\int_{S^1}  (v f'')^2 + 8 (v f'') (v' f' ) +  6 (v' f' )^2 dx -\lambda \int_{S^1}v f^2 dx \geqslant 0.
\end{equation*}
By completing the squares we have 
\begin{small}
\begin{equation*}
2\int_{S^1} \lbrace (v f'')^2 + 8 (v f'') (v' f' ) +  6 (v' f' )^2 \rbrace dx -\lambda \int_{S^1}v f^2 dx \geqslant \int_{S^1} \lbrace f''^2 v^2 - 52 f'^2 v'^2 \rbrace dx -\lambda \int_{S^1}v f^2 dx.
\end{equation*}
\end{small}
The lower bound on the density $v>m$ yields
\begin{equation*}
\int_{S^1} \lbrace f''^2 v^2 - 52 f'^2 v'^2 \rbrace dx -\lambda \int_{S^1}v f^2 dx \geqslant  m^2 \int_{s^1} f''^2 dx - 52 \int_{s^1} f'^2 u'^2 dx - m \lambda \int_{s^1} f^2 dx.
\end{equation*}
Hölder's inequality and energy bound $E(u)<c$ imply
\begin{equation*}
m^2 \int_{s^1} f''^2 dx - 52 \int_{s^1} f'^2 u'^2 dx - m \lambda \int_{s^1} f^2 dx \geqslant m^2 \int_{s^1} f''^2 dx - 52 c |f'^2 |_{\infty} - m \lambda \int_{s^1} f^2 dx.
\end{equation*}
By reordering and absorbing the constants in $\lambda$, the energy is $\lambda$-convex along $\nu_s$ at $\nu$ if $\forall f \in C^{\infty}(S^1)$ we have
\begin{equation}\label{w02}
|f''|_{L^2}^2 - \alpha |f'|_{\infty}^2 - \lambda |f|_{L^2}^2 \geqslant 0
\end{equation}
where $\alpha=\frac{52 c}{m^2}$. By Lemma \ref{Maximum bound} the claim has been proved.\\

Now consider the energy sub-level set $E_c$. By Theorem \ref{theo:Uniform con.} Wasserstein convergence implies uniform convergence on $E_c$. Therefore there exists a $\delta=\delta_c$ such that we have $v > m$ for all $\nu=vdx \in E_c \cap B_{\delta}(\mu)$. Assume that $\nu_0 , \nu_1 \in E_c \cap B_{\delta}(\nu)$. Let $\nu_s$ be the geodesic connecting $\nu_0$ to $\nu_1$. By Lemma \ref{lem interpolant regularity} there exist $\hat{m}$ and $\hat{c}$ depending only on $c$ and $m$ such that $E(\nu_s)< \hat{c} < +\infty$ and $v_s> \hat{m}>0$. By the argument at the beginning of the proof there exists a $\hat{\lambda}=\hat{\lambda}_{m,c}$ such that $E$ is $\hat{\lambda}$-convex along the geodesic $\nu_s$. The constant $\hat{\lambda}$ is uniform for all pairs of smooth measures inside $E_c\cap B_{\delta}(\mu)$. Therefore, by Lemma \ref{lem smooth app.} $E$ is restricted $\hat{\lambda}$-convex at $\mu$.
\end{proof}\\

\begin{corollary}
The gradient flow trajectory of the Dirichlet energy on $S^1$ with a positive initial data exists and is unique at least for a short period of time.
\end{corollary}

\begin{corollary}
The positive periodic solutions of the thin-film equation $\partial_t u = - \partial_x (u \partial_x^3 u)$ are locally well-posed.
\end{corollary}


\section{Other classes of equations}

In this section, we show that the theory developed in the last two sections can be applied to a wide class of energy functionals and evolution equations of higher order and different forms. Note that the result of Theorem \ref{theorem1} is general and it can be applied to any energy functional, provided that it is restricted $\lambda$-convex. The corresponding lemmas from Section 3 for the energies studied here can be derived in a similar fashion with minor modifications. Hence, we discuss the proofs only briefly.\\


\subsection{Higher order equations}
The family that we study here is of the form $E(u)=\frac{1}{2}\int_{S^1}|u^{(k)}|^2dx$ for $k \in \mathbb{N}$. The flow of this family of energies  corresponds to the solution of the higher order non-linear equations of the form $\partial_t u =(-1)^k \partial_x(u\partial_x^{2k+1}u)$.\\

Consider $u \in D(E)$. Finiteness of $|u|_{H^k}$ in particular implies that the $H^1$-norm of $u$ is bounded. Since we only used the $H^1$-norm bounds in Lemmas \ref{theo:Uniform con.}, \ref{lem:l.s.c}, and \ref{lem smooth app.}, they automatically follow for this class of energies. Therefore, Wasserstein and uniform convergence are equivalent on energy sub-level sets, $E$ is lower semi continuous, and one can use approximation by smooth functions to study convexity.\\

In Lemma \ref{lem interpolant regularity} we derived bounds on $T''$ by taking derivatives of the explicit formula of $T'$ given by the Monge–Ampère equation. In the same fashion, one can find bounds on higher derivatives of the optimal map by taking more derivatives of the Monge–Ampère equation. For generalization of Lemma \ref{Maximum bound}, we have to show that $\forall \alpha ~ \exists \lambda$ such that
\[|f^{(m+1)}|_{L^2}^2 -\alpha |f^{(m)}|_{\infty}^2 -\lambda |f|_{L^2}^2 \geqslant 0\]
for every smooth vector field $f$. By induction assume that for any $\alpha_m>0$ there exists $\lambda_m \leqslant 0$ such that
\begin{equation}\label{FFF}
|f^{(m)}|_{\infty}^2 \leqslant \frac{1}{\alpha_m}|f^{(m+1)}|_{L^2}^2 - \lambda_m|f|_{L^2}^2.
\end{equation}
Let $\alpha_{m+1}>0$ be given. By applying Lemma \ref{Maximum bound} to $f^{(m+1)}$, there exists a $\hat{\lambda} \leqslant 0$ such that
\begin{equation}\label{u01}
|f^{(m+1)}|_{\infty}^2\leqslant \dfrac{1}{2\alpha_{m+1}}|f^{(m+2)}|_{L^2}^2 - \hat{\lambda}|f^{(m)}|_{L^2}^2.
\end{equation}
Put $\alpha_{m}= -2\hat{\lambda}$, by (\ref{FFF}) there exists $\lambda_m \leqslant 0$ such that
\begin{equation*}
|f^{(m)}|_{\infty}^2 \leqslant \frac{-1}{2\hat{\lambda}}|f^{(m+1)}|_{L^2}^2 - \lambda_m|f|_{L^2}^2.
\end{equation*}
$|f^{(m)}|_{L^2}^2 \leqslant |f^{(m)}|_{\infty}^2$ on $\R/\Z$. Therefore
\begin{equation*}
-\hat{\lambda}|f^{(m)}|_{L^2}^2 \leqslant \frac{1}{2}|f^{(m+1)}|_{L^2}^2 +  \hat{\lambda} \lambda_m|f|_{L^2}^2.
\end{equation*}
Plugging into (\ref{u01}) yields
\begin{equation*}
|f^{(m+1)}|_{\infty}^2 \leqslant \dfrac{1}{2\alpha_{m+1}}|f^{(m+2)}|_{L^2}^2 + \dfrac{1}{2}|f^{(m+1)}|_{L^2}^2 + \hat{\lambda} \lambda_m|f|_{L^2}^2.
\end{equation*} 
Therefore
\[ |f^{(m+1)}|_{\infty}^2 - \dfrac{1}{2}|f^{(m+1)}|_{L^2}^2 \leqslant \dfrac{1}{2\alpha_{m+1}}|f^{(m+2)}|_{L^2}^2 + \hat{\lambda} \lambda_m|f|_{L^2}^2. \]
By $|f^{(m+1)}|_{L^2}^2 \leqslant |f^{(m+1)}|_{\infty}^2$ and by setting $\lambda_{m+1}=\frac{-1}{2} \hat{\lambda} \lambda_m$, we have
\begin{equation}\label{GGG}
\forall \alpha_{m+1}>0 ~~ \exists \lambda_{m+1} \leqslant 0 ~~~s.t.~~~ |f^{(m+1)}|_{\infty}^2 \geqslant \frac{1}{\alpha_{m+1}}|f^{(m+2)}|_{L^2}^2 - \lambda_{m+1}|f|_{L^2}^2.
\end{equation}

In conclusion, all the Lemmas in the previous section can be applied to higher order energies. We now study convexity of the energies along smooth vector fields on a measure $\mu=udx$ with positive density $u>m$ and finite energy $E(u)<c<\infty$.
\begin{equation*}
\begin{aligned}
\dfrac{d^2}{ds^2} |_{s=0} E(u_s) &= \dfrac{d^2}{ds^2} |_{s=0} \int_{S^1}(\partial_y^k u_s(y))^2 dy \\
&= \dfrac{d^2}{ds^2} |_{s=0} \int_{S^1} \lbrace (\dfrac{\partial x}{\partial y}\dfrac{\partial }{\partial x})^k \dfrac{u(x)}{1+sf'(x)} \rbrace^2 \frac{\partial y}{\partial x} dx  .
\end{aligned}
\end{equation*}
Since we study all the different orders at the same time, we consider the general form given by a polynomial $P$ which is determined by the order of the energy. We have
\[ \dfrac{d^2}{ds^2} |_{s=0} E(u_s)= \int_{S^1} |uf^{(k)}|^2 + P(u,u^{(1)},...,u^{(k)};f,f^{(1)},...,f^{(k-1)}) dx\]
where $P$ is of order at most 2 with respect to each of its entries, and the order of the derivative of each term in $P$ is at most $k$. At a measure with positive density and finite energy, we have $u>m$ and $|u^{(i)}|_{\infty}< M$ for all $i < k$ where $M$ depends only on $E(u)$. Also $|f^{(i)}|_{L^2} \leqslant |f^{(k-1)}|_{\infty}$ for all $i<k-1$. Therefore similar to the calculation of Theorem \ref{theorem1}, for positive constants $\beta_1, \beta_2$ we have
\[ \dfrac{d^2}{ds^2} |_{s=0} E(u_s) \geqslant \beta_1 |f^{(k)}|_{L^2}^2 - \beta_2 |f^{(k-1)}|_{\infty}^2. \] 
Therefore $E$ is convex at $u$ if we can find $\lambda$ such that
\begin{equation*}
\beta_1 |f^{(k)}|_{L^2}^2 - \beta_2 |f^{(k-1)}|_{\infty}^2 \geqslant \lambda \int_{S^1} u f^2 dx.
\end{equation*}
This implies that the energy is $\lambda$-convex at $u$ because by (\ref{GGG}) for $\alpha=\alpha_{m,c}$ there exists $\lambda$ such that 
\begin{equation*}
|f^{(k)}|_{L^2}^2 -\alpha|f^{(k-1)}|_{\infty}^2  - \lambda |f|_{L^2}^2 \geqslant 0 ~~~~ \forall f \in C^{\infty}(S^1)
\end{equation*}
Hence we have proved the following theorem.

\begin{theorem}
The energies of the form
\begin{equation*}
E(u)=\begin{cases} \int_{S^1}|\partial_x^k u(x)|^2dx   &\mu =udx,~u\in H^k(S^1), \\
+\infty            &\text{else.}
\end{cases}
\end{equation*}
are restricted $\lambda$-convex on the positive measures with finite energy. In particular, periodic gradient flow solutions of
\begin{equation*}
\partial_t u = (-1)^k \partial_x (u \partial_x^{2k+1} u)
\end{equation*}
with positive initial data exist and are unique for a short time. 
\end{theorem}


\subsection{Different forms of equations}

Consider the energies of the form $E(u)=\int_{S^1} g(u,\partial_x u) dx$. We start by calculating the second derivative of the energy along a geodesic induced by a vector field $f\in C^{\infty}(S^1)$.

\begin{equation*}
\begin{aligned}
\dfrac{d^2}{ds^2} |_{s=0} E(u_s) &= \dfrac{d^2}{ds^2} |_{s=0} \int_{S^1} g(u_s(y),\partial_y u_s(y)) dy \\
&= \dfrac{d^2}{ds^2} |_{s=0} \int_{S^1}  g\left( \dfrac{u(x)}{1+sf'(x)},\dfrac{\partial x}{\partial y}\dfrac{\partial }{\partial x} (\dfrac{u(x)}{1+sf'(x)})\right) \frac{\partial y}{\partial x} dx\\
&= \dfrac{d^2}{ds^2} |_{s=0} \int_{S^1}  g\left( \dfrac{u(x)}{1+sf'(x)},\dfrac{1}{1+sf'(x)} \partial_x (\dfrac{u(x)}{1+sf'(x)})\right) (1+sf'(x)) dx
\end{aligned}
\end{equation*} 
where we used the change of variable $y=x+sf(x)$. Therefore we have
\begin{equation} \label{matrix}
\dfrac{d^2}{ds^2}\vert_{s=0}E(\mu_s)=\int_{s^1} \left[\begin{array}{ccc}
f' & f''
\end{array} \right] A
\left[
\begin{array}{c}
 f' \\
 f''
\end{array}
\right] dx
\end{equation}
where the matrix $A$ is given by
\begin{equation*}
A=\left[
\begin{array}{ccc}
2 u' g^{(0,1)}+4 u'^2 g^{(0,2)}+4 u u' g^{(1,1)}+u^2 g^{(2,0)} & \quad 2 u g^{(0,1)}+2 u u' g^{(0,2)}+u^2 g^{(1,1)} \\
2 u g^{(0,1)}+2 u u' g^{(0,2)}+u^2 g^{(1,1)} & u^2 g^{(0,2)}
\end{array}
\right]
\end{equation*}

Note that if $A$ is positive definite, then the energy is convex. We study the class of the form $g(u,\partial_x u)= |\partial_x (u^a)|^2 $ with $a>0$. Finiteness of the energy implies that $u^a$ is $C^{0,\frac{1}{2}}$ continuous with modulus of continuity smaller than the energy. Because $a>0$, $u$ is continuous and since $\int_{S^1}udx=1$, there exists a point $x_0$ with $u(x_0)=1$. Without loss of generality we assume $x_0=0$. We have
\begin{equation*}
|u(x)^a-1|\leqslant  \sqrt{c|x|} ~ \Rightarrow ~  u(x) \leqslant (1+\sqrt{c|x|})^{\frac{1}{a}}.
\end{equation*}  
Therefore there exists a uniform $M<\infty$ such that $u<M$ for all $u\in E_c$. We now briefly discuss the corresponding lemmas from Section 3.\\

\textbf{Equivalence of Wasserstein and uniform convergence on energy sub-level sets.} Let $u_2(x_0)-u_1(x_0)>h$. Then we have $u_2^a(x)\geqslant u_2^a(x_0)-\sqrt{c|x-x_0|}$ and $u_1^a(x) \leqslant u_1^a(x_0)+\sqrt{c|x-x_0|}$. Therefore the star-like shape in Lemma \ref{theo:Uniform con.} should be replaced by a modified version, given by $(u_2^a(x_0)-\sqrt{c|x-x_0|})^{\frac{1}{a}}$ and $(u_1^a(x_0)+\sqrt{c|x-x_0|})^{\frac{1}{a}}$, and the rest of the proof goes similarly. Hence, we have equivalence of the Wasserstein and uniform convergence on the energy sub-level sets.\\

\textbf{Lower semi-continuity and smooth approximation.} Having Lemma \ref{theo:Uniform con.} for this class of energies, the proof of Lemmas \ref{lem:l.s.c} and \ref{lem smooth app.} can be repeated by replacing $u$ with $u^a$. Hence the energy $E(u)=\int_{S^1}|\partial_x u^a|dx$ is lower semi continuous and one can use approximation by smooth functions.\\

\textbf{Energy of the interpolant.} Let $u$ be bounded away from zero $u>m>0$. When $a\geqslant 1$
\begin{equation}\label{s01}
m^{2(a-1)}\int_{S^1}|\partial_x u|^2dx ~~ \leqslant ~ E(u)=\int_{S^1}u^{2(a-1)}|\partial_x u|^2 dx ~~ \leqslant ~ M^{2(a-1)}\int_{S^1}|\partial_x u|^2dx
\end{equation}
and when $0<a<1$
\begin{equation}\label{s08}
M^{2(a-1)}\int_{S^1}|\partial_x u|^2dx ~~ \leqslant ~~ E(u)=\int_{S^1}u^{2(a-1)}|\partial_x u|^2 dx ~~ \leqslant ~~ m^{2(a-1)}\int_{S^1}|\partial_x u|^2dx.
\end{equation}
By equivalence of Wasserstein and uniform convergence, there exists $\delta$ such that $v>m$ for all $v\in B_{\delta}(u)\cap E_c$. Also we have proved that $v<M$ for all $v \in E_c$. Therefore we can refer to the calculation for the Dirichlet energy and just compare the energy of the geodesic with the corresponding Dirichlet energy using (\ref{s01}) and (\ref{s08}) to find a bound on the energy of interpolate points along a geodesic.\\

In conclusion, all of the required lemmas are true. By (\ref{matrix}), along a geodesic induced by a smooth vector field $f$ we have 
\begin{equation*}
\begin{aligned}
\dfrac{d^2}{ds^2} |_{s=0} E(u_s) &= \int_{S^1}2a^2u^{2(a-1)}\left( (uf'')^2 +4(1+a) (uf'')(u'f') + (1+a)(1+2a)(u'f')^2 \right) dx\\
&\geqslant \int_{S^1}\alpha_1 (u^a f'')^2 -\alpha_2((u^a)'f')^2  dx
\end{aligned}
\end{equation*}
for some constants $\alpha_1, \alpha_2$. Similar to (\ref{w02}), we have  
\begin{equation*}
\dfrac{d^2}{ds^2} |_{s=0} E(u_s) - m \lambda \int_{S^1}uf^2dx ~~ \geqslant ~~  |f''|^2_{L^2} -\alpha|f'|^2_{\infty} -\lambda |f|^2_{L^2} ~~ \geqslant ~~ 0
\end{equation*}
where the last inequality follows from Lemma \ref{Maximum bound}. We have proved the following theorem.

\begin{theorem}
For every $a>0$ 
\begin{equation*}
E(u)=\begin{cases} \int_{S^1}|\partial_x u(x)^a |^2dx   &\mu =udx,~u \in H^1(S^1) \\
+\infty            &\text{else.}
\end{cases}
\end{equation*}
is restricted $\lambda$-convex on positive measures with finite energy. In particular, periodic gradient flow solutions of
\begin{equation*}
\partial_t u = -2a \partial_x (u \partial_x(u^{a-1} \partial_x^2 u^a))
\end{equation*}
with positive initial data exist and are unique for a short time. 
\end{theorem}

An interesting example is the \textbf{Fisher Information}
\[E(u)= \frac{1}{2} \int_{S^1}|\partial_x u(x)^{\frac{1}{2}} |^2dx\]
which corresponds to the \textbf{quantum drift diffusion Equation} 
\[\partial_t u =-\partial_x (u \partial_x \frac{\partial_x^2 \sqrt{u}}{\sqrt{u}}).\] 
Therefore we have local well-posedness of periodic solutions of the quantum drift diffusion equation with positive initial data.\\

Another interesting case is the limiting case $a=0$. The corresponding energy can be written as $E(u)=\frac{1}{2}\int_{S^1}|\partial_x \log u|^2 dx $. Finiteness of the energy result in $C^{0,\frac{1}{2}}$ continuity of $\log u$. All of the lemmas can be repeated in a similar fashion for this energy. Furthermore, finiteness of the energy implies a lower bound for the measure because
\begin{equation}
|\log u(x)- \log (1)|\leqslant \sqrt{c|x|} ~~\Longrightarrow ~~ e^{-\sqrt{c}} \leqslant u(x) \leqslant e^{\sqrt{c}}. 
\end{equation}
Therefore positivity is preserved along the flow. By (\ref{matrix}) we have
\begin{equation*}
\frac{d^2}{ds^2}|_{s=0}E(\mu_s)=\int_{S^1}2(f'')^2+4(f'')(\frac{u'}{u}f')-2(\frac{u'}{u}f')^2 dx.
\end{equation*}
By \ref{Maximum bound} there exists $\lambda$ such that
\begin{equation*}
\frac{d^2}{ds^2}|_{s=0}E(\mu_s) ~ \geqslant ~ |f''|^2_{L^2}-6 c e^{\sqrt{c}} |f'|^2_{\infty}- \lambda |f|_{L^2} ~ \geqslant ~ 0.
\end{equation*}
Hence $E$ is restricted $\lambda$ convex at $u \in E_c$. Furthermore, since there is a uniform lower bound $e^{-\sqrt{c}}$ for all $v\in E_c$, the constant $\lambda$ is uniformly bounded along the flow. Therefore, the gradient flow is globally well-posed and we have the following theorem.

\begin{theorem}
Wasserstein gradient flow of the energy 
\begin{equation*}
E(u)=\begin{cases} \frac{1}{2}\int_{S^1}|\partial_x \log u|^2 dx  &\mu =udx,~u \in H^1(S^1) \\
+\infty            &\text{else.}
\end{cases}
\end{equation*}
is globally well-posed. Hence the equation
\[\partial_t u =-\partial_x (u \partial_x^2 \frac{\partial_x u}{u^2} )\]
with periodic boundary condition is well-posed.
\end{theorem}

\textbf{Remarks.} There are some simple and some more challenging directions to extend the developed method to other classes of equations. As  a simple application, one can construct other classes of restricted $\lambda$-convex functionals by combining the ones already studied. For example, the solution of the energy $E(u)=\int_{S^1} \lbrace |\partial_x u|^2 + \epsilon \frac{1}{u^2} \rbrace dx$, which is the Dirichlet energy with a perturbation, is globally well-posed. The reason is that the second term forces the energy to remain positive. One interesting problem is the analysis of equations in higher dimensions. Our method is utilizing Sobolev embedding theorem on energy sub-level sets which is getting weaker on higher dimensions. An interesting question is whether it is possible to solve this problem with studying higher order energies.\\

\textbf{Acknowledgements.} I wish to express my gratitude to my supervisor Almut Burchard for her guidance and support throughout this project. Her generosity with her energy and time will not be forgotten. I am also grateful to Robert McCann, Dejan Slep{\v{c}}ev, and Nicola Gigli for all the insightful discussions.\\


\bibliographystyle{plain}
\bibliography{gradient}

\begin{thebibliography}{10}

\bibitem{AmbrosioBook}
Luigi Ambrosio, Nicola Gigli, and Giuseppe Savar{\'e}.
\newblock {\em Gradient flows in metric spaces and in the space of probability
  measures}.
\newblock Lectures in Mathematics ETH Z\"urich. Birkh\"auser Verlag, Basel,
  second edition, 2008.

\bibitem{AmbrosioZamboti}
Luigi Ambrosio, Giuseppe Savar{\'e}, and Lorenzo Zambotti.
\newblock Existence and stability for {F}okker-{P}lanck equations with
  log-concave reference measure.
\newblock {\em Probab. Theory Related Fields}, 145(3-4):517--564, 2009.

\bibitem{WassersteinDiffusion}
Sebastian Andres and Max-K. von Renesse.
\newblock Particle approximation of the {W}asserstein diffusion.
\newblock {\em J. Funct. Anal.}, 258(11):3879--3905, 2010.

\bibitem{Benamou}
Jean-David Benamou and Yann Brenier.
\newblock A computational fluid mechanics solution to the {M}onge-{K}antorovich
  mass transfer problem.
\newblock {\em Numer. Math.}, 84(3):375--393, 2000.

\bibitem{Bertozzi}
A.~L. Bertozzi and M.~Pugh.
\newblock The lubrication approximation for thin viscous films: regularity and
  long-time behavior of weak solutions.
\newblock {\em Comm. Pure Appl. Math.}, 49(2):85--123, 1996.

\bibitem{Carrillo}
Adrien Blanchet, Vincent Calvez, and Jos{\'e}~A. Carrillo.
\newblock Convergence of the mass-transport steepest descent scheme for the
  subcritical {P}atlak-{K}eller-{S}egel model.
\newblock {\em SIAM J. Numer. Anal.}, 46(2):691--721, 2008.

\bibitem{BrenierPolar}
Yann Brenier.
\newblock Polar factorization and monotone rearrangement of vector-valued
  functions.
\newblock {\em Comm. Pure Appl. Math.}, 44(4):375--417, 1991.

\bibitem{GangboCarlen}
E.~A. Carlen and W.~Gangbo.
\newblock Constrained steepest descent in the 2-{W}asserstein metric.
\newblock {\em Ann. of Math. (2)}, 157(3):807--846, 2003.

\bibitem{Mccann02}
Jos{\'e}~A. Carrillo, Robert~J. McCann, and C{\'e}dric Villani.
\newblock Contractions in the 2-{W}asserstein length space and thermalization
  of granular media.
\newblock {\em Arch. Ration. Mech. Anal.}, 179(2):217--263, 2006.

\bibitem{Slepcev}
Jos{\'e}~A. Carrillo and Dejan Slep{\v{c}}ev.
\newblock Example of a displacement convex functional of first order.
\newblock {\em Calc. Var. Partial Differential Equations}, 36(4):547--564,
  2009.

\bibitem{Numerics02}
Fausto Cavalli and Giovanni Naldi.
\newblock A {W}asserstein approach to the numerical solution of the
  one-dimensional {C}ahn-{H}illiard equation.
\newblock {\em Kinet. Relat. Models}, 3(1):123--142, 2010.

\bibitem{DeGiorgiMM}
Ennio De~Giorgi.
\newblock New problems on minimizing movements.
\newblock In {\em Boundary value problems for partial differential equations
  and applications}, volume~29 of {\em RMA Res. Notes Appl. Math.}, pages
  81--98. Masson, Paris, 1993.

\bibitem{Numerical}
Bertram D{\"u}ring, Daniel Matthes, and Josipa~Pina Mili{\v{s}}i{\'c}.
\newblock A gradient flow scheme for nonlinear fourth order equations.
\newblock {\em Discrete Contin. Dyn. Syst. Ser. B}, 14(3):935--959, 2010.

\bibitem{Evans}
Lawrence~C. Evans.
\newblock {\em Partial differential equations}, volume~19 of {\em Graduate
  Studies in Mathematics}.
\newblock American Mathematical Society, Providence, RI, second edition, 2010.

\bibitem{McCannKim}
Alessio Figalli, Young-Heon Kim, and Robert McCann.
\newblock Regularity of optimal transport maps on multiple products of spheres.
\newblock {\em To appear in J. Eur. Math. Soc. (JEMS)}.

\bibitem{JKO}
Richard Jordan, David Kinderlehrer, and Felix Otto.
\newblock The variational formulation of the {F}okker-{P}lanck equation.
\newblock {\em SIAM J. Math. Anal.}, 29(1), 1998.

\bibitem{Loeper}
Gr{\'e}goire Loeper.
\newblock Regularity of optimal maps on the sphere: the quadratic cost and the
  reflector antenna.
\newblock {\em Arch. Ration. Mech. Anal.}, 199(1):269--289, 2011.

\bibitem{MTW}
Xi-Nan Ma, Neil~S. Trudinger, and Xu-Jia Wang.
\newblock Regularity of potential functions of the optimal transportation
  problem.
\newblock {\em Arch. Ration. Mech. Anal.}, 177(2):151--183, 2005.

\bibitem{McCann-Thesis}
Robert~John McCann.
\newblock {\em A convexity theory for interacting gases and equilibrium
  crystals}.
\newblock ProQuest LLC, Ann Arbor, MI, 1994.
\newblock Thesis (Ph.D.)--Princeton University.

\bibitem{StickyParticles}
Luca Natile and Giuseppe Savar{\'e}.
\newblock A {W}asserstein approach to the one-dimensional sticky particle
  system.
\newblock {\em SIAM J. Math. Anal.}, 41(4):1340--1365, 2009.

\bibitem{OttoPorous}
Felix Otto.
\newblock The geometry of dissipative evolution equations: the porous medium
  equation.
\newblock {\em Comm. Partial Differential Equations}, 26(1-2):101--174, 2001.

\bibitem{VillaniTopics}
C{\'e}dric Villani.
\newblock {\em Topics in optimal transportation}, volume~58 of {\em Graduate
  Studies in Mathematics}.
\newblock American Mathematical Society, Providence, RI, 2003.

\bibitem{Villani}
C{\'e}dric Villani.
\newblock {\em Optimal transport, Old and New}, volume 338 of {\em Grundlehren
  der Mathematischen Wissenschaften}.
\newblock Springer-Verlag, Berlin, 2009.

\end{thebibliography}

\end{document}